\documentclass[letter,11pt]{amsart}
\usepackage[latin1]{inputenc}   
\usepackage{amssymb,amsmath,amsthm,mathrsfs,mathdots}
\usepackage{graphicx,mathpazo}
\usepackage[all]{xy}
\usepackage[usenames,dvipsnames]{color}
\usepackage{enumitem}

\usepackage{youngtab,ytableau}
\ytableausetup{centertableaux,boxsize=0.25em}

\usepackage[colorlinks=true,linkcolor=Cerulean,pagebackref,citecolor=Mahogany]{hyperref}

 \usepackage{tikz}
\usetikzlibrary{arrows,decorations.pathmorphing,decorations.pathreplacing,positioning,shapes.geometric,shapes.misc,decorations.markings,decorations.fractals,calc,patterns}

\tikzset{>=stealth',
     cvertex/.style={circle,draw=black,inner sep=1pt,outer sep=3pt},
     vertex/.style={circle,fill=black,inner sep=1pt,outer sep=3pt},
     star/.style={circle,fill=yellow,inner sep=0.75pt,outer sep=0.75pt},
     tvertex/.style={inner sep=1pt,font=\criptsize},
     gap/.style={inner sep=0.5pt,fill=white}}

\newcommand{\ZZ}{\ensuremath{\mathbb{Z}}}

\newcommand{\A}{\ensuremath{\mathbb{A}}} 
\newcommand{\St}{\ensuremath{\mathcal{S}}} 

\DeclareMathOperator{\cchar}{char}

\DeclareMathOperator{\Sing}{Sing}
\DeclareMathOperator{\Spec}{Spec}

\newcommand{\mc}[1]{\ensuremath{\mathcal{#1}}}   

\newcommand{\cA}{\mathcal{A}}
\newcommand{\cF}{\mathcal{F}}
\newcommand{\cL}{\mathcal{L}}

\newcommand{\cX}{\mathcal{X}}
\newcommand{\x}{{\boldsymbol x}} 
\newcommand*\atwo{ \tikz[baseline=(char.base),->,>=stealth',shorten >=1pt,auto,node distance=1.2cm, thick,main node/.style={circle,draw}]{
		\node[shape=circle,draw,inner sep=2pt] (char) {1};
		\node[shape=circle,draw,inner sep=2pt,right of=char] (2) {2};
		\path
		(char)	edge (2);}}

\newcommand*\aTwo{ \tikz[baseline=(char.base),->,>=stealth',shorten >=1pt,auto,node distance=1.2cm, thick,main node/.style={circle,draw}]{
		\node[shape=circle,draw,inner sep=2pt] (char) {1};
		\node[shape=circle,draw,inner sep=2pt,right of=char] (2) {2};
		\path
		(2)	edge (char);}}

\newcommand{\hoehe}{\textcolor{white}{\bigg|}}
\newcommand{\Jac}{\operatorname{Jac}}
\newcommand{\car}{\operatorname{char}}

\newcommand{\change}[1]{{\color{black} #1}}

\newcounter{CountAlpha}

\theoremstyle{theorem}

\newtheorem{MainThm}[CountAlpha]{Theorem}  

\newtheorem{MainCor}[CountAlpha]{Corollary}

\newtheorem{Thm}{Theorem}[section]         
\newtheorem{lemma}[Thm]{Lemma}

\newtheorem{corollary}[Thm]{Corollary}
   
\newtheorem{prop}[Thm]{Proposition}
\newtheorem{proposition}[Thm]{Proposition} 

\newtheorem{Qu}[Thm]{Question}

\theoremstyle{definition}
\newtheorem{defi}[Thm]{Definition} 

\newtheorem{example}[Thm]{Example}
\newtheorem{Bem}[Thm]{Remark}
\newtheorem*{Not}{Notation}

\title[Classification of singularities of cluster algebras of finite type]{Classification of singularities of cluster algebras of finite type: the case of trivial coefficients}

\author{Angelica Benito}
\address{Departamento de Did\'acticas Espec\'ificas,
	Facultad de Formaci\'on de Profesorado, 
	Universidad Aut\'onoma de Madrid,
	28049 Madrid, Spain}
\email{angelica.benito@uam.es} 

\author{Eleonore Faber}
\address{
School of Mathematics, University of Leeds, LS2 9JT Leeds, UK
}
\email{e.m.faber@leeds.ac.uk}

\author{Hussein Mourtada}
\address{Universit\'e de Paris, Sorbonne Universit{\'e},
	Institut Math\'ematique de Jussieu-Paris Rive Gauche,
	B\^atiment Sophie Germain, case 7012,
	75205 Paris Cedex 13, France}
\email{hussein.mourtada@imj-prg.fr}

\author{Bernd Schober}
\address{
Carl von Ossietzky Universit\"at Oldenburg,
Institut f\"ur Mathematik,
Ammerl\"ander Heerstra{\ss}e 114 - 118,
26129 Oldenburg (Oldb), Germany
}
\email{bernd.schober@uni-oldenburg.de}

\date{May 30, 2022}

\makeatletter
\@namedef{subjclassname@2020}{\textup{2020} Mathematics Subject Classification}
\makeatother


\thanks{Competing interests: The authors declare none.}  

\begin{document}

\maketitle

\begin{abstract}
	We provide a complete classification of the singularities of cluster algebras of finite type with trivial coefficients. 
	Alongside, we develop a constructive desingularization of these singularities via blowups in regular centers over fields of arbitrary characteristic.  
	Furthermore, from the same perspective, 
	we study a family of cluster algebras which are not of finite type
	and which arise from a star shaped quiver. 
\end{abstract}

\section{Introduction}

Cluster algebras were originally introduced by Fomin and Zelevinsky to study total positivity phenomena and Lusztig's canonical bases in Lie theory,
see e.g.~\cite{Fomin-ICM, Leclerc-ICM, LusztigCanonical}.
They quickly developed to a vibrant research area going far beyond its initial motivations, 
and with connections to many other areas,
such as 
algebraic geometry \cite{BFMN, BMRS2015,GHK,Nagao,MS16,Scott06},
commutative algebra \cite{GLS13,M13,M14},
combinatorics \cite{RS,FominReading}, 
representation theory of finite dimensional algebras and quivers \cite{MarshReinekeZelevinsky,BMRRT06,GLSInventiones,DWZ},
higher Teichm\"uller 
spaces
 \cite{FG,GSV},
or mirror symmetry \cite{GHKK,GS,KS}. 
For further connections and applications see e.g. the surveys \cite{KellerDerivedSurvey, WilliamsSurvey}.
\\
A {\em cluster algebra} is a subring of the field of rational functions in $ n $ variables over a base field $ K $%
\footnote{In fact, one could also work over more general bases, e.g. $ \ZZ $ instead of $ K $, see \cite[Section~5]{FZ2002}
		but we restrict our attention to fields.}.  
It is a commutative ring that is constructed  differently than most rings that usually are considered in commutative algebra: 
instead of generators and relations, one starts with a set of distinguished generators (the \emph{cluster variables}) and then iteratively constructs (via the process of \emph{mutation}) all other generators of the ring. 
In this article, we will mostly consider cluster algebras $\mathcal{A}(Q)$ that are constructed from a \emph{quiver} $Q$. 
We will also assume that $Q$ is totally mutable, that is, we assume trivial coefficients.
For the precise definitions and 
an outline of the more general construction via skew-symmetrizable matrices we refer to Section~\ref{sec:ClusterBasics}.
\\ 
Our main theme here is to investigate cluster algebras from the perspective of singularity theory, in particular, resolution of singularities. 
Our studies were motivated by an interesting coincidence in classifications: 
on the one hand, cluster algebras $\cA(Q)$ of finite type are classified by ADE-Dynkin diagrams \cite{FZ2003II}, 
whereas on the other hand the dual resolution graphs of the Kleinian surface singularities are classified by the same diagrams, see \cite{ArtinRational, BrieskornMathAnn1966, Lipman1969}, 
as well as simple hypersurface singularities in the sense of Arnold \cite{ArnoldNormalForms1972}. 
For an overview, see e.g., \cite{SlodowySLN815}. 
Thus we were guided by the following 

\begin{Qu} Let $\mc{A}$ be a cluster algebra. Which types of singularities can $\Spec(\mc{A})$ have? Can one classify these singularities for certain types of cluster algebras?
\end{Qu}

\begin{Qu} How can one describe resolutions of singularities of cluster algebras and do these resolutions take into account the combinatorial structure of the cluster algebras?
\end{Qu}

So far, there are only few results in this direction. 
In \cite{BMRS2015}, 
Benito, Muller, Rajchgot, and Smith proved that 
locally acyclic cluster algebras are strongly $ F $-regular 
(when defined over a field of prime characteristic)
and 
that they have at worst canonical singularities (over a field of characteristic $0$). 
Further, Muller, Rajchgot, and Zykoski \cite{MRZ2018} showed 
that the lower bound cluster algebra 
(which is an approximation of a given cluster algebra obtained by a suitable truncation of the construction process) 
is Cohen-Macaulay and normal.

In this paper we study \emph{cluster algebras of finite type}, which
can be classified in terms of Dynkin diagrams (as mentioned above, finite type cluster algebras from quivers are of type ADE, and more generally, all cluster algebras of finite type are classified by the crystallographic Coxeter groups  \cite{FZ2003II}). 
We provide a complete classification of their singularities and describe their embedded desingularization in the case of trivial coefficients.
Due to the combinatorial nature of cluster algebras,
the characteristic of the base field $ K $
does not play an essential role. 

\begin{Not}
	For a Dynkin diagram $ X_n \in \{ A_{n_1}, B_{n_2}, C_{n_3}, D_{n_4}, E_6, E_7, E_8, F_4, G_2 \mid n_i \geq i \} $,
	we denote by $ \cA(X_n) $ the corresponding cluster algebra with trivial coefficients.
	Note that the corresponding variety $ \Spec(\cA(X_n)) $ is a 
	different object to what is called a cluster variety.
	The latter will not play a role 
	in the present work.
\end{Not}

Let us briefly introduce notions in the context of simple singularities,
which we need to state our classification theorem. 
For the entire list of simple singularities in arbitrary characteristics, we refer to \cite[Definition~1.2]{GreuelKroening}.
Let $ K $ be a field of arbitrary characteristic. 
A formal power series $ f \in K[[x,y,z]] $ 
is of type $ A_m $,
for some $ m \in \ZZ_{\geq 1} $, if $ K[[x,y,z]]/ \langle f \rangle $ is isomorphic to 
$ K[[x,y,z]]/ \langle z^{m+1} + xy \rangle $.
Note that if $ K $ is algebraically closed and $ \car(K) \neq 2 $,
then we may perform a change in the variables such that $ z^{m+1} + xy = z^{m+1}+ \widetilde{x}^2 + \widetilde{y}^2 $. 
\\
Let $ n \in \ZZ $ with $ n \geq 3 $. 
A formal power series $ f \in K[[z, x_1, \ldots, x_n]] $ is 
of type $ A_1 $ if 
$ K[[z, x_1, \ldots, x_n]]/ \langle f \rangle $ is isomorphic to $ K[[z, x_1, \ldots, x_n]]/ \langle g \rangle $, where
\[
	g = \begin{cases}
	
	z^2 + x_1 x_2 + \cdots + x_{n-1} x_n
	& 
	\mbox{if } n \equiv 0 \mod 2,
	
	\\
	
	z x_1 + x_2 x_3 + \cdots + x_{n-1} x_n
	
	& 
	\mbox{if } n \equiv 1 \mod 2  .
	\end{cases}
\]
Let $ N > n \geq 2 $. 
We say that an $ n $-dimensional variety $ X \subset \A_K^N $
with an isolated singularity at a closed point $ x $ 
is locally isomorphic to an
isolated hypersurface singularity of type $ A_1 $
(resp.~ of type $ A_2 $ if $ n = 2 $ and $ N = 3 $),
if the completion of the local ring of $ X $ at $ x $ is isomorphic to
$ K[[z_0, x_1, \ldots, x_n]]/ \langle f \rangle $,
where $ f $ is a power series of type $ A_1 $ (resp.~of type $ A_2 $). 
If $ \dim(\Sing(X)) \geq 1 $,
we say that $ X $ is locally at some $ U \subseteq  \Sing(X) $ isomorphic to a cylinder over an isolated hypersurface singularity of type $ A_1 $ in $ \A_K^{m+1} $,
for some $ m < n $,
if locally at $ U $, 
$ X $ is isomorphic to a hypersurface in $ \A_K^{n+1} $ defined by $ f(z, x_1, \ldots, x_m) = 0 $,
where $ f \in K[z, x_1, \ldots, x_m ] $ is of type $ A_1 $. 
Furthermore, in the local situation, 
we say that a {regular} hypersurface $ V(h) $ is transversal to the cylinder,
if $( z, x_1, \ldots, x_m) $ do not appear in $ h $ after a suitable change in $ (x_{m+1}, \ldots, x_n) $. 
\\
In fact, for the cases which we consider, 
we do not need to pass to the completion
since we may construct a suitable change of variables
already after localizing.

Using the introduced notions, 
we can state our main result
on the classification of singularities of cluster algebras of finite type with trivial coefficients. 

{\samepage
\begin{MainThm}
	\label{Thm:A}
	Let $ K $ be a field of characteristic $ p \geq 0 $.
	\begin{enumerate}
		\item 
		$ \Spec(\cA(A_n) )$, $ n \geq 2 $, is singular 
		if and only if $ p \neq 2 $ and $ n \equiv 3 \mod 4 $, or if $ p = 2 $ and $ n \equiv 1 \mod 2 $. 
		In the singular case, $ \Spec( \cA(A_n) ) $ is locally isomorphic to an isolated hypersurface singularity of type $ A_1 $.
		
		\smallskip 
		
		\item 
		$ \Spec(\cA(B_n)) $, $ n \geq 2 $, is singular 
		if and only if $ p \neq 2 $ and $ n \equiv 3 \mod 4 $, or if $ p = 2 $.
		In the singular case, $ \Spec( \cA(B_n) ) $ is locally isomorphic to an isolated hypersurface singularity of type $ A_1 $.
		
		\smallskip  
		
		\item 
		$ \Spec(\cA(C_n)) $, $ n \geq 3 $, is singular 
		if and only if $ p = 2 $.
		In the singular case, we have 
		\[  
			\Sing(\Spec( \cA(C_n) )) \cong \Spec(\cA (A_{n-2})) . 
		\]
		\begin{enumerate}
			\item 
			If $ n \equiv 0 \mod 2 $, then $ \Sing(\Spec( \cA(C_n) ))  $ is {regular} and
			$ \Spec( \cA(C_n) ) $ is locally isomorphic to a 
			cylinder over an isolated hypersurface singularity of type $ A_1 $ in $ \A_K^3 $.  
			
			\item 
			If $ n \equiv 1 \mod 2 $ and $ n > 3 $, 
			then $ \Sing(\Spec( \cA(C_n) ))  $ has an isolated singularity of type $ A_1 $ at the origin
			and,
			locally at the origin,
			$ \Spec(\mathcal{A}(C_n)) $ is 
			isomorphic to a hypersurface of the form
			(where $ n = 2m +1 $)
			\[  
			\Spec ( k[x_1, \ldots, x_{2m}, y, z] / \langle \, yz + \big( \sum_{i=1}^{m} x_{2i-1} x_{2i} \big)^2  \ \rangle 
			\ , 
			\]
			while at a singular point different from the origin,
			$ \Spec(\mathcal{A}(C_n)) $ is locally isomorphic to a 
			cylinder over an isolated hypersurface singularity  of type $ A_1 $ in $ \A_K^3 $.
			
			\item
			If $ n = 3 $,
				then $ \Sing(\cA(C_n)) $ is isomorphic to two lines intersecting transversally at the origin. 
				All other statements of (2) remain true for $ m = 1 $.
		\end{enumerate}

		\smallskip 
		
		\item 
		\begin{enumerate}
			\item 
			$ \Spec(\cA(D_4)) $ is isomorphic to a subvariety of $ \A_K^6 $ and $ \Sing(\Spec(\cA(D_4))) $ consists of the 6 coordinate axes.
			At the origin, $ \Spec(\cA(D_4)) $ is locally isomorphic to the intersection of two hypersurface singularities of type $ A_1 $,
			while at a singular point different from the origin,
			$ \Spec(\mathcal{A}(D_4)) $ is locally isomorphic to a 
			cylinder over an isolated hypersurface singularity of type $ A_1 $ in $ \A_K^4 $ intersected with a {regular} hypersurface which is transversal to the cylinder.
			
			\item 
			If $ p \neq 2 $ and $ n \not\equiv 0 \mod 4 $
			or if $ p = 2 $ and $ n \equiv 1 \mod 2 $,
			then the singular locus of $ \Spec(\cA(D_n)) $
			has a single irreducible component $ Y_0 $, 
			which is {regular} and of dimension $ n - 3 $.
			Moreover, $ \Spec(\cA(D_n)) $ is locally at the singular locus isomorphic to a cylinder over a hypersurface singularity of type $ A_1 $ in $ \A_K^4 $.
			
			\item 
			Let $ n > 4 $.
			If $ p \neq 2 $ and $ n \equiv 0 \mod 4 $
			or if $ p = 2 $ and $ n \equiv 0 \mod 2 $,
			then $ \Sing(\Spec(\cA(D_n))) = Y_0 \cup \bigcup_{i=1}^4 Y_i $,
			where $ Y_i $ are isomorphic to coordinate axes, for $ i \geq 1 $,
			and $ Y_0 $ is irreducible, singular at the origin, and of dimension $ n - 3 $.
			At the origin, $ \Spec(\cA(D_n)) $ is locally isomorphic to the intersection of two hypersurface singularity of type $ A_1 $,
			while $ Y_0 $ is locally isomorphic to a hypersurface singularity of type $ A_1 $.
			Away from the origin, 
			the situation is analogous to the two $ D_n $-cases before. 
		\end{enumerate}
	
		\smallskip 
		
		\item 
		$ \Spec(\cA(E_7)) $ is singular if and only if $ p = 2 $. 
		In the singular case, $ \Sing (\Spec(\cA(E_7))) $ is a {regular} surface and locally at the singular locus, $ \Spec(\cA(E_7)) $ is
		isomorphic to a 
			cylinder over an isolated hypersurface singularity of type $ A_1 $ in $ \A_K^6 $ intersected with a {regular} hypersurface which is transversal to the cylinder.

		\smallskip
		
		\item 
		$ \Spec(\cA(G_2)) $ is singular if and only if $ p = 3 $. 
		In the singular case, $ \Spec( \cA(G_2) ) $ is locally isomorphic to an isolated hypersurface singularity of type $ A_2 $ in $ \A_K^3 $.
		
		\smallskip 
		
		\item 
		The varieties corresponding to the cluster algebras 
		$ \cA(E_6), \cA(E_8), $ and $ \cA(F_4) $ are {regular}.  
	\end{enumerate}
\end{MainThm}
}
\newpage

Cluster algebras of finite type arise in applications very often with
non-trivial coefficients.
The presence of frozen variables (i.e., directions in which one cannot mutate) can affect the existence and type of singularities. 
Therefore an interesting question would be to extend the above classification in the case of non-trivial coefficients. 
This is the subject of further future studies.

Part (1) of Theorem~\ref{Thm:A} has previously been proven for $ p \neq 2 $ in \cite[Proposition~A.1]{MRZ2018}.
Note that the statement in loc.~cit.~is characteristic free, 
but the special case $ \car(K) = 2 $ has been overseen. 
\\
We note that from our classification follows that there is no obvious direct link between the singularities of cluster algebras of finite types and the rational double point singularities. 
For example, for cluster algebras of type ADE, only hypersurface singularities of type $A_1$ (cluster algebras of type $A$) or more complicated configurations (cluster algebras of type $D$) appear.

As a consequence of Theorem~\ref{Thm:A}, we can construct an
embedded resolution of singularities for cluster algebras of finite type.

\begin{MainCor}
	\label{Cor:B}
	Let $ K $ be any field
	and let $ X := \Spec (\cA ) $, 
	where  $\cA $ is any cluster algebra of finite type. 
	There exists a finite sequence $ \pi $ 
	of blowups in {regular} centers such that the strict transform of $ X $ along $ \pi $ is {regular} and it has simple normal crossings with the exceptional divisors.  
\end{MainCor}

In order to prove Theorem~\ref{Thm:A},
we choose first a suitable presentation of the cluster algebra $ \cA(X_n) $,
which arises from an acyclic seed.
The latter has the benefit that the cluster algebra can be described as a quotient of a polynomial ring in $ 2n $ variables over $ K $ 
by an ideal generated by $ n $ relations determined by the initial seed, 
the {\em exchange relations}.
\\
We determine the singular locus by applying
Zariski's criterion for regularity \cite[Theorem~11, p.~39]{Zar47}.
	The latter is a variant of
	the Jacobian criterion for smoothness
	\cite[\S 2.2]{Cut2004},
	where derivatives with respect to a fixed $ p $-basis of the base field $ K $ have to be taken into account in the Jacobian matrix.
	Since the coefficients appearing in the exchange relations are contained in $ \ZZ $,
	we do not have to consider a $ p $-basis of $ K $.
	In particular, $ K $ can be any field and is not necessarily perfect.
\\
 In general, it is not very pleasant to handle the maximal minors of a matrix of size $ n \times 2n $. 
Via subtle eliminations of variables, 
we deduce from the mentioned presentation a new one, 
which is better suited for our task.
In particular, the number of generators in the resulting set diminishes to at most three and often only one. 
From this, we can then detect and classify the singularities of the corresponding variety and thus of $ \Spec ( \cA(X_n) ) $. 
\\
A key ingredients in our studies are continuant polynomials,
as they naturally appear in the elimination process. 
Therefore, as a preparation for Theorem~\ref{Thm:A},
we examine them from a perspective of singularity theory in Section~\ref{S:continuants}.

Furthermore, we also take a look beyond cluster algebras of finite type. 
More precisely, we investigate the singularities 
of a class of cluster algebras 
which arise from a star shaped quiver $ \St_n $, where $ n \geq 2 $:
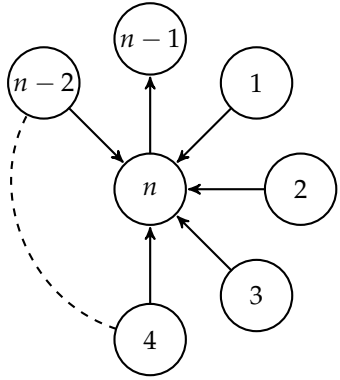
\begin{center} 
	\begin{tikzpicture}[->,>=stealth',shorten >=1pt,auto,node distance=2cm, thick,main node/.style={circle,draw,white}]

	\node[main node] (n) {1 1};
	\node[main node] (1) [above right of=n] {1 1};
	\node[main node] (2) [right of=n] {1 1};
	\node[main node] (3) [below right of=n] {1 1};
	\node[main node] (4) [below of=n] {1 1};
	\node[main node] (n-2) [above left of=n] {1 1};
	\node[main node] (n-1) [above  of=n] {1 1};

	\node at (n) {\small $n$};
	\node at (1) [above right of=n] {\small $ 1 $};
	\node at (2) [right of=n] {\small $2$};
	\node at (3) [below right of=n] {\small $3$};
	\node at (4) [below of=n] {\small $4$};
	\node at (n-2) [above left of=n] {{\small $n-2$}};
	\node at (n-1) [above  of=n] {{\small $n-1$}};

	\draw (n) circle (13.5pt);
	\draw (1) circle (13.5pt);
	\draw (2) circle (13.5pt);
	\draw (3) circle (13.5pt);
	\draw (4) circle (13.5pt);
	\draw (n-2) circle (13.5pt);
	\draw (n-1) circle (13.5pt);

	\path

	(3) edge (n)
	(1) edge (n)
	(2) edge (n)
	(4) edge (n)
	(n-2) edge (n)
	(n) edge (n-1);

	\path[dashed]
	(4) edge[bend left=50,-] (n-2);
	
	\end{tikzpicture}
\end{center} 

Observe that the case $ n \leq 4 $ has already been treated in Theorem~\ref{Thm:A} 
since the corresponding quivers are coming from the Dynkin diagrams
$ A_2 $, $ A_3 $, and $ D_4 $, respectively.

\begin{MainThm}
	\label{Thm:C}
	Let $ K $ be any field and $ n \geq 4 $.
	Let $ \cA(\St_n) $ be the cluster algebra over $ K $ arising form the star shaped quiver $ \St_n $. 
	The singular locus $ \Sing ( \Spec(\mc{A}(\St_n))) $ has 
	$ (n-1)(n-2)2^{n-4} $ irreducible components, 
	where each of them is {regular} and of dimension $ n - 3 $.	
	Locally at a generic point of such a component, 
	$ \Spec (\cA(\St_n) ) $ is isomorphic to an $ A_1 $-hypersurface singularity.
	On the other hand, locally at the closed point determined by the intersection of all these components, 
	$ \change{ \Spec(\cA(\St_n)) } $ is isomorphic to a toric variety,
	defined by the binomial ideal
	\[
		\langle x_1 x_2 - x_{2k-1} x_{2k} \mid k \in \{ 2, \ldots, n - 1 \} \rangle \subset K [x_1, \ldots, x_{2n-2}]_{ \langle x_1, \ldots, x_{2n-2} \rangle }.
	\]
	The singularities of $ \Spec(\cA(\St_n)) $ are resolved by first separating the irreducible components of its singular locus and then blowing up their strict transforms.
\end{MainThm}

The appearing integer sequence $ ( (n-1)(n-2)2^{n-4} )_{n \geq 4 } $
can be found in the {\em The On-Line Encyclopedia of Integer Sequences}, 
\cite[Sequence A001788]{OEIS}.
In Remark~\ref{R:Starcomp}, we explain the connection to one of the descriptions given in loc.~cit.

\medskip

Let us give a brief summary of the contents: 
In Section~\ref{sec:ClusterBasics}, 
we recall basic notions and results on cluster algebras.
In particular, we address the classification of finite type via Dynkin diagrams.
After that we study the singularities of continuant polynomials in Section~\ref{S:continuants}, 
as they play an essential role in our investigations.
Then, we show Theorem~\ref{Thm:A} and Corollary~\ref{Cor:B} by studying case-by-case the cluster algebras $ \cA(X_n) $ of different Dynkin types
\change{in Sections~\ref{sec:Quiver} (quiver case) and~\ref{sec:Matrix} (non-quiver case).}
We end with the proof of Theorem~\ref{Thm:C} in Section~\ref{sec:star}.

{\em Acknowledgments:}
The authors want to thank Bernhard Keller for comments on an earlier version of this manuscript. \change{We also thank the anonymous referee for helpful comments to improve the exposition.} This project was initiated during a three weeks stay at the Mathematisches Forschungsinstitut Oberwolfach in the context of an Oberwolfach Research Fellowship in 2020.
All authors are grateful to the Institute for their support and hospitality during their visit. 

\bigskip

\section{Cluster Algebras basics}
\label{sec:ClusterBasics}

Since we do not require that the reader is familiar with the theory of cluster algebras, 
we first briefly recall the basics on cluster algebras associated to quivers 
and the necessary notions to deal with all cluster algebras of finite type. 
That is, we also outline the more general theory using skew-symmetrizable matrices. 
However, for most of the paper, we will be dealing with cluster algebras associated to quivers, 
so we provide a more detailed exposition for this case. 
For more details on the general theory, we refer the reader to literature,
\cite{FZ2002,FZ2003II,FWZ2016,FWZ2017,FWZ2020}.

\smallskip 

A {\em quiver} $ Q $
is a finite directed graph.
So, $ Q = (Q_0, Q_1) $  is a 
pair of two finite sets, 
where $ Q_0 = \{ 1, \ldots, n \} $ is the set of vertices and 
$ Q_1 $ is the set of arrows between the vertices. 
An element of $ Q_1 $ can be identified with a pair $ (i,j) $
with $ i,j \in Q_0 $,
where the corresponding arrow goes from $ i $ to $ j $;
we also write $ i \to j $. 
Note that multiple arrows are allowed between {two} vertices. Additionally, we always assume:

\begin{itemize}
	\item 
	$ Q $ does not contain any loops,
	i.e., $ (i,i) \notin Q_1 $ for all $ i \in Q_0 $. 
	
	\item 
	There are no oriented $ 2 $-cycles in $ Q $,
	i.e., if $ (i,j) \in Q_1 $, then $ (j,i) \notin Q_1 $. 
\end{itemize}

For example, the pictures of the quivers $ Q = (\{1,2,3\}, \{ (1,2)_1, (1,2)_2, (3,2) \}) $
and $ Q' := (\{ 1, 2, 3, 4 \}, \{ (1,2), (3,2), (4,2) \} )  $ are:

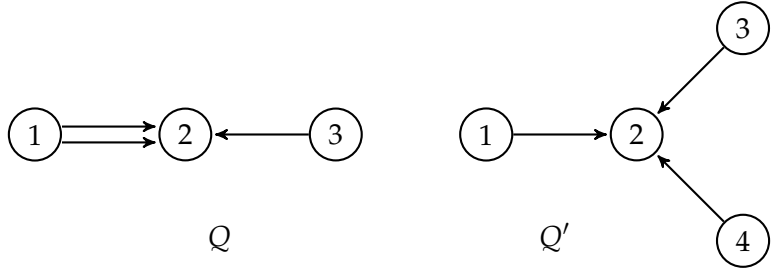
\begin{center} 
	\begin{tikzpicture}[->,>=stealth',shorten >=1pt,auto,node distance=2cm, thick,main node/.style={circle,draw}]
	
	\node[main node] (1) {$1$};
	\node[main node] (2) [right of=1] {$2$};
	\node[main node] (3) [right of=2] {$ 3 $};
	
	\path
	([yshift=-3pt,xshift=0pt]1.east) edge  ([xshift=0pt,yshift=-3pt]2.west)
	([yshift=3pt,xshift=0pt]1.east) edge  ([xshift=0pt,yshift=3pt]2.west)
	(3) edge (2);

	\node[main node] (1a) [right of=3] {$1$};
	\node[main node] (2a) [right of=1a] {$2$};
	\node[main node] (3a) [above right of=2a] {$ 3 $};
	\node[main node] (4a) [below right of=2a] {$ 4 $};
	
	\path
	(1a) edge (2a) 
	(3a) edge (2a)
	(4a) edge (2a);
	\end{tikzpicture}
	\\[-0.75cm]
	$ Q $ 
	\hspace{111pt}
	$ Q' $
	\\[0.75cm]
\end{center}

(In the set of arrows for $ Q $, we wrote $ (1,2)_\alpha $, for $ \alpha  \in \{ 1, 2 \} $, in order to indicate that there are two different arrows from $ 1 \to 2 $ appearing in $ Q $.)

\begin{defi}
	Let $ Q = ( Q_0, Q_1 ) $ be a quiver
	and $ k \in Q_0 $ be a vertex.
	The {\em quiver mutation $ \mu_k $} (in direction $ k $) 
	transforms $ Q $ into a new quiver $ Q' = \mu_k (Q) $, 
	which is obtained in the following way:
	\begin{enumerate}[leftmargin=*,label=(\arabic*)]
		\item 
		for every directed path $ i \to k \to j $ in $ Q $,
		we add a new arrow $ i \to j $;
		
		\item 
		we reverse the arrows incident to the vertex $ k $;
		
		\item 
		we remove oriented $ 2 $-cycles until there is none left. 
	\end{enumerate}
	
	Two quivers $ Q^{(1)} $ and $ Q^{(2)} $ are called {\em mutation-equivalent},
	if there exists a sequence of mutations transforming $ Q^{(1)} $ into a quiver $ Q' $, which is isomorphic to $ Q^{(2)} $
	(i.e., there exists a bijection $ f \colon Q'_0 \to Q^{(2)}_0 $ between the set of vertices such that $ (i,j) \in Q'_1 $ if and only if $ (f(i), f(j)) \in Q^{(2)}_1 $). 
	If this is the case, we write $ Q^{(1)} \sim Q^{(2)} $.
\end{defi}

Let us illustrate the mutation procedure for an example.
Here, we mutate at the vertex $ k = 1  $. 
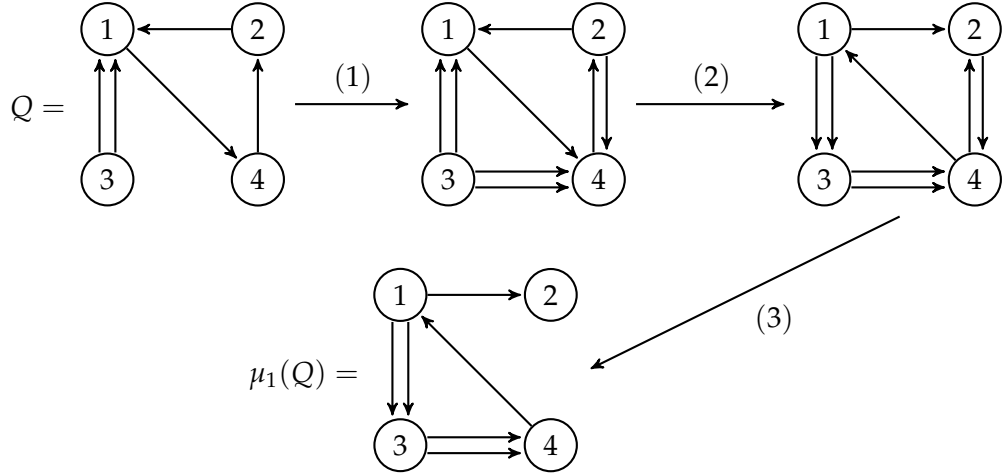
\begin{center} 
	\begin{tikzpicture}[->,>=stealth',shorten >=1pt,auto,node distance=5cm, thick]
	
	\node (A) {
		\raisebox{1.2cm}{$ Q = $
		}
		\begin{tikzpicture}[->,>=stealth',shorten >=1pt,auto,node distance=2cm, thick,main node/.style={circle,draw},mutate node/.style={circle,fill=berndGruen,draw},frozen/.style={draw}]
		
		\node[main node] (k) {$1$};
		\node[main node] (i) [right of=k] {$2$};
		\node[main node] (j) [below of=k] {$3$};
		\node[main node] (l) [below of=i] {$4$};

		\path 
		(i) edge (k)
		(l) edge (i)
		(k) edge (l)
		([xshift=-3pt]j.north) edge ([xshift=-3pt]k.south)
		([xshift=3pt]j.north) edge ([xshift=3pt]k.south);
		\end{tikzpicture}
	};
	
	\node (B) [right of=A] {
		\begin{tikzpicture}[->,>=stealth',shorten >=1pt,auto,node distance=2cm, thick,main node/.style={circle,draw},mutate node/.style={circle,fill=berndGruen,draw},frozen/.style={draw}]
		
		\node[main node] (k) {$1$};
		\node[main node] (i) [right of=k] {$2$};
		\node[main node] (j) [below of=k] {$3$};
		\node[main node] (l) [below of=i] {$4$};

		\path 
		(i) edge (k)
		(k) edge (l)
		([xshift=-3pt]j.north) edge ([xshift=-3pt]k.south)
		([xshift=3pt]j.north) edge ([xshift=3pt]k.south)
		([xshift=-2pt]l.north) edge ([xshift=-2pt]i.south)
		([xshift=3pt]i.south) edge ([xshift=3pt]l.north)
		([yshift=-3pt]j.east) edge ([yshift=-3pt]l.west)
		([yshift=3pt]j.east) edge ([yshift=3pt]l.west)
		;
		\end{tikzpicture}
	};

	\node (C) [right of=B]{
		\begin{tikzpicture}[->,>=stealth',shorten >=1pt,auto,node distance=2cm, thick,main node/.style={circle,draw},mutate node/.style={circle,fill=berndGruen,draw},frozen/.style={draw}]
		
		\node[main node] (k) {$1$};
		\node[main node] (i) [right of=k] {$2$};
		\node[main node] (j) [below of=k] {$3$};
		\node[main node] (l) [below of=i] {$4$};

		\path 
		(k) edge (i)
		(l) edge (k)
		([xshift=-3pt]k.south) edge ([xshift=-3pt]j.north)
		([xshift=3pt]k.south) edge ([xshift=3pt]j.north)
		([xshift=-2pt]l.north) edge ([xshift=-2pt]i.south)
		([xshift=3pt]i.south) edge ([xshift=3pt]l.north)
		([yshift=-3pt]j.east) edge ([yshift=-3pt]l.west)
		([yshift=3pt]j.east) edge ([yshift=3pt]l.west)
		;
		\end{tikzpicture}
	};

	\node (D) [below right of=A] {
		\raisebox{1.2cm}{$ \mu_1(Q) = $
		}
		
		\begin{tikzpicture}[->,>=stealth',shorten >=1pt,auto,node distance=2cm, thick,main node/.style={circle,draw},mutate node/.style={circle,fill=berndGruen,draw},frozen/.style={draw}]
		
		\node[main node] (k) {$1$};
		\node[main node] (i) [right of=k] {$2$};
		\node[main node] (j) [below of=k] {$3$};
		\node[main node] (l) [below of=i] {$4$};

		\path 
		(k) edge (i)
		(l) edge (k)
		([xshift=-3pt]k.south) edge ([xshift=-3pt]j.north)
		([xshift=3pt]k.south) edge ([xshift=3pt]j.north)
		([yshift=-3pt]j.east) edge ([yshift=-3pt]l.west)
		([yshift=3pt]j.east) edge ([yshift=3pt]l.west)
		;
		\end{tikzpicture}
	}; 
	
	\path 
	(A) edge node {$(1)$} (B)
	(B) edge node {$(2)$} (C) 
	(C.south) edge node {$(3)$} (D.east) ;
\end{tikzpicture}
\end{center}

\begin{Bem}
	\label{Rk:frozen}
In general, one subdivides the set of vertices into two disjoint sets: 
the mutable vertices, for which we are allowed to perform a mutation,
and frozen vertices, which cannot be mutated,
see \cite[Section~2.1]{FWZ2016}. 
In this paper, we only deal with quivers where all vertices are mutable, so we will not go into details of frozen variables.
\end{Bem}

From now on, we fix a field
$ K $ and a field $ \cF $, which is isomorphic to the field of rational functions over $ K $ in $ n $ variables. 

\begin{defi}
A {\em labeled seed of geometric type} in $ \cF $ is a pair $ (\x, Q ) $,
where
\begin{itemize}
	\item 
	$ \x = (x_1, \ldots, x_n) $ is a $ n $-tuple of algebraically independent elements and such that $ \cF \cong K(x_1, \ldots, x_n) $;
	
	\item 
	$ Q $ is a quiver with $ n $ vertices, which neither contains loops nor $ 2 $-cycles. 
\end{itemize}
The $ n $-tuple $ \x $ is called the {\em cluster} of the seed 
and $ x_1, \ldots, x_n $ are the {\em cluster variables}.
The number $ n $ of vertices is called the {\em rank} of the seed. 
\end{defi}

Since all seeds appearing in this article are labeled seeds of geometric type, 
we simply speak of {\em seeds} in $ \cF $.

The mutation of a quiver extends in the following way to a seed. 

\begin{defi}
Let $ ( \x , Q ) $ be a seed in $ \cF $ and let $ k \in Q_0 = \{ 1, \ldots, n \}  $. 
The {\em seed mutation} $ \mu_k $ (in direction $ k $)
transforms $ ( \x, Q ) $ into a new seed $ \mu_k ( \x , Q ) = (\x', Q') $,
which is obtained in the following way:
\begin{itemize}
	\item 
	$ Q' = \mu_k ( Q) $;
	
	\item 
	$ \x' = (x_1', \ldots, x_n') $,
	where $ x_j' = x_j $ for $ j \neq k $ and 
	$ x_k' \in \cF $ is the element determined by the 
	{\em exchange relation}
	\begin{equation}
	\label{eq:ex_rel}
	x_k x_k' 
	= \prod_{i \rightarrow k } x_i + \prod_{i \leftarrow k } x_i 
	\end{equation}
\end{itemize}

Two seeds $ (\x^{(1)}, Q^{(1)}) $ and $ (\x^{(2)}, Q^{(2)}) $ are called {\em mutation-equivalent},
if there exists a sequence of mutations transforming one seed into the other
(up to permutation of the cluster variables, which also induces an isomorphism of quivers). 
If this is the case, we write $ (\x^{(1)}, Q^{(1)}) \sim (\x^{(2)}, Q^{(2)}) $.  
\end{defi}

Note that $ \mu_k $ is an involution, i.e., $ \mu_k (\mu_k (\x, Q)) = (\x, Q) $.
On the other hand, 
there exist examples for which $ \mu_k (\mu_\ell (\x, Q)) \neq \mu_\ell (\mu_k (\x, Q)) $, where $ \ell \neq k $.
For example, one can verify that $ \mu_3 ( \mu_1 (Q)) \neq \mu_1 (\mu_3 (Q)) $ in the example given above.

\begin{defi} \label{def:clusteralgebra}
Let $ (\x, Q) $ be a seed in $ \cF $. 
We set
\[
\cX := \cX (\x, Q) := \bigcup_{ (\x',Q') \sim (\x, Q) } \x'.
\]
The {\em cluster algebra $ \cA := \cA(\x, Q) $} (of geometric type, over $ K $) determined by the seed $ (\x, Q) $
is defined as the sub-$ K $-algebra of $ \cF $ generated by all cluster variables, 
\[
\cA(\x, Q) := K [\cX] .
\] 
\end{defi}

\begin{Bem}[{cf.~\cite[Section~3.1]{FWZ2016}}] 
	\label{Bem:Mat}
The data of $Q$ can be encoded in an $n \times n$ integer matrix $ B = B(Q) $ with entries
{$ b_{i,j} $, which are equal to the number of arrows $ i \to j $ in $ Q $ and
where an arrow $ j \to i $ is counted with negative sign for $ b_{i,j} $,
i.e.,}
\[ 
	b_{ij} := \# \{ \text{ arrows } (i,j) \in Q_1\} - \#\{ \text{ arrows } (j,i) \in Q_1 \} \ . 
\]	
Then $B$ is called the \emph{exchange matrix}. 
Note that $B(Q)$ is skew-symmetric and that $B$ determines $Q$, 
so that sometimes the notion $(\x, B)$ for the seed $(\x,Q)$ is used. 
Moreover, mutation $\mu_k(Q)$ can also be defined on the matrix $B$, where the mutation
{$ B' := \mu_k (B) $} of $B$ {in direction $ k $} is given by
\begin{equation} \label{eq:matrixmutation}
b'_{ij} : =\begin{cases} -b_{ij} & \text{ if } i=k \text{ or } j=k , \\
b_{ij}+\frac{1}{2}(|b_{ik}|b_{kj} +b_{ik}|b_{kj}|) & \text{ otherwise. }
\end{cases} 
\end{equation}

More generally, {the notion of a seed $ ( \x, B ) $ and its corresponding cluster algebra $ \cA(\x,B) $ can be extended to the following setting:
	
\begin{itemize}
	\item 
	$ B := (b_{i,j})_{i,j \in \{ 1, \ldots, n \}} $ a 
	\emph{skew-symmetrizable} integer matrix,
	i.e., there exists a diagonal matrix $ D $ 
	with integer entries such that $ D B $ is skew-symmetric,
	
	\item 
	where the mutation rule $ \mu_k(B) $ is given by \eqref{eq:matrixmutation}, and 
	
	\item 
	the exchange relations \eqref{eq:ex_rel} become
	\begin{equation}
	\label{exchangeskewsymmetrizable}
		{x_k x_k'
			=
			\prod_{\stackrel{j=1}{b_{j,k} >0}}^n x_j^{b_{j,k}} + \prod_{\stackrel{j=1}{b_{j,k} <0}}^n x_j^{-b_{j,k}}}
		\ .
	\end{equation}   
\end{itemize}
}

{Let us point out that the sign pattern of a skew-symmetrizable matrix is skew-symmetric.
	The sign pattern of a such a matrix can be encoded in terms of a quiver.
	More precisely, to a skew-symmetrizable matrix $ B $,}
{one associates the quiver $\Gamma(B)$ in the following way: 
	if $b_{i,j}>0$, then we have an arrow} 
 	{$ i \to j $}.
 	{Note that if $B$ is skew-symmetric, then $B=B(Q)$ and $\Gamma(B)$ is the quiver $Q$ with multiple arrows collapsed into $1$.} 
\end{Bem}

\smallskip

\begin{example}
Let us discuss the first non-trivial example. 
Consider the seed $ (\x, Q) $,
where $ \x = (x_1, x_2 ) $ and 
$ Q $ is the quiver with two vertices and one arrow between them,  
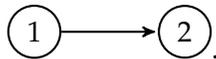
\begin{center} 
	\begin{tikzpicture}[->,>=stealth',shorten >=1pt,auto,node distance=2cm, thick,main node/.style={circle,draw}]
	
	\node[main node] (1) {$1$};
	\node[main node] (2) [right of=1] {$2$};
	
	\path
	(1) edge (2);
	\end{tikzpicture}.
\end{center}
The following table describes the behavior of $ (\x, Q ) $ along repeated mutation:
\[
\begin{array}{|r|c|c|}
\hline 
\mbox{quiver} \hoehe \hspace{1.5cm} & \mbox{cluster} & \mbox{expression in terms of initial cluster $ \x = (x_1, x_2) $}
\\
\hline \hline 
Q = \atwo 
\hoehe 
& 
(x_1, x_2) 
& -
\\
\hline 
\mu_1 (Q) = \aTwo 
\hoehe 
& 
(x_1^{(1)}, x_2^{(1)})
&
x_1^{(1)} = \dfrac{1+ x_2}{x_1},
\hspace{15pt}
x_2^{(1)} = x_2
\\
\hline 
\mu_2 (Q) = \aTwo 
\hoehe 
& 
(x_1^{(2)}, x_2^{(2)})
&
x_1^{(2)} = x_1, 
\hspace{15pt} 
x_2^{(2)} = \dfrac{x_1 + 1}{x_2}
\\
\hline 
\mu_1(\mu_1 (Q)) = \atwo 
\hoehe 
& 
(x_1^{(11)}, x_2^{(11)})
&
x_1^{(11)} = x_1, 
\hspace{15pt} 
x_2^{(11)} = x_2
\\
\hline 
\mu_2(\mu_2 (Q)) = \atwo 
\hoehe 
& 
(x_1^{(22)}, x_2^{(22)})
&
x_1^{(22)} = x_1, 
\hspace{15pt} 
x_2^{(22)} = x_2
\\
\hline 
\mu_1(\mu_2 (Q)) = \atwo 
\hoehe 
& 
(x_1^{(12)}, x_2^{(12)})
&
x_1^{(12)} = \dfrac{1+x_1+x_2}{x_1 x_2}, 
\hspace{15pt} 
x_2^{(12)} = \dfrac{x_1 + 1}{x_2}
\\
\hline 
\mu_2(\mu_1 (Q)) = \atwo 
\hoehe 
& 
(x_1^{(21)}, x_2^{(21)})
&
x_1^{(21)} = \dfrac{1 + x_2}{x_1}, 
\hspace{15pt} 
x_2^{(21)} = \dfrac{1+x_1+x_2}{x_1 x_2}
\\ 
\hline 
\change{\mu_2(\mu_1(\mu_2 (Q))) = \aTwo} 
\hoehe 
& 
\change{(x_1^{(212)}, x_2^{(212)})}
&
\change{x_1^{(212)} = \dfrac{1+x_1+x_2}{x_1 x_2}, 
\hspace{15pt} 
x_2^{(212)} = \dfrac{1 + x_2}{x_1}} 
\\ 
\hline 
\end{array} 
\] 
\change{Notice that $ x_1^{(212)} = x_2^{(21)} $ and $ x_2^{(212)} = x_1^{(21)} $.}
Therefore, we have 
\[
\cA:= 
\cA (\x, Q) 
= K \left[x_1, x_2, \frac{1+x_1}{x_2}, \frac{1+x_2}{x_1}, \frac{1+x_1+ x_2}{x_1 x_2} \right]
=
K \left[x_1, \frac{1+x_1}{x_2}, \frac{1+x_2}{x_1} \right]
\cong 
\]\[	\cong
K[u,v,w]/ \langle uvw - u - v - 1 \rangle
,  
\] 
where the second equality holds since
\[ 
x_2 = \dfrac{x_2 + 1}{x_1} \cdot x_1 - 1   
\ \ 
\mbox{ and } 
\ \ 
\dfrac{1+x_1+ x_2}{x_1 x_2} = \dfrac{x_2 + 1}{x_1} \cdot \dfrac{x_1 + 1}{x_2} - 1 .
\]
Observe that the singular locus of $ \Spec (\cA) $ is empty. 
\end{example}

In the example above, all cluster variables can be expressed as Laurent polynomials in the initial cluster variables $ x_1, x_2 $.
Indeed, this is always true by \cite[Theorem~3.1]{FZ2002}.
In our context this can be stated as follows: 

\begin{Thm}[Laurent {phenomenon}]
Let $ (\x, Q) $ be a seed in $ \cF $.  
Every cluster variable can be expressed as a Laurent polynomial with integer coefficients in $ \x $. 
\end{Thm}

It can be quite tedious to determine all seeds mutation-equivalent to a given initial seed $ ( \x, Q) $. 
A useful tool for determining mutation equivalent quivers and related invariants is the  {\tt Java} applet \cite{Keller-Java-app}. \\
Sometimes these calculations can be avoided by considering the lower cluster algebra, 
which can be easily determined and which coincides with the cluster algebra in many interesting cases, see Theorem~\ref{Thm:clus=low} below. 
Note that \change{the} following results 
({Definition}~\ref{Def:lowerC}, Lemma \ref{L:BFZ_Cor1.17}, {Theorem}~\ref{Thm:clus=low}) 
also hold in the skew-symmetrizable case, i.e., for $\cA(\x,B)$ where $B$ is skew-symmetrizable.

\begin{defi} \label{Def:lowerC}
Let $ (\x, Q) $ be a seed in $ \cF $. 
The {\em lower bound cluster algebra $ \cL(\x, Q) $} of $ (\x,Q) $ is defined as 
\[
\cL(\x,Q)
:=K [x_1, \ldots, x_n, x_1', \ldots, x_n'],
\]
where $ x_1', \ldots, x_n' $ are the elements that we obtain by the exchange relation~\eqref{eq:ex_rel} after mutating $ Q $ once in direction $ 1, \ldots, n $, respectively. 
\end{defi}

We immediately see that the inclusion
\[  
\cL(\x,Q) \subseteq \cA(\x,Q) 
\] 
holds
and whenever we have equality, 
then it is easy to provide a set of generators for $ \cA(\x,Q)  $. 

Let $ J $ be the ideal of relations among the generators of $ \cL(\x,Q) $.
Clearly, the exchange relations provide the elements 
$ x_k x_k' - \prod_{i \rightarrow k } x_i - \prod_{i \leftarrow k } x_i  \in J $, for $ k \in \{ 1, \ldots, n \} $. 
In general, it may happen that these are not all relations between the generators, see \cite[subsection~1.2]{MRZ2018}.
Nonetheless, the following useful result holds for acyclic quivers. 
Recall that a quiver is called {\em acyclic}, if it does not contain an oriented cycle. 
In the case $(\x,B)$ we say that $\cA(\x,B)$ is \emph{acyclic} if $\Gamma(B)$ is an acyclic quiver.

\begin{lemma}[cf.~{\cite[Corollary 1.17]{BFZ2005}}]
	\label{L:BFZ_Cor1.17}
	If $ Q $ is acyclic,
	then the exchange relations ~\eqref{eq:ex_rel}
	generate the ideal $ J$  of relations among the generators of $ \cL(\x,Q) $.
	Moreover, the polynomials $ x_k x_k' - \prod_{i \rightarrow k } x_i - \prod_{i \leftarrow k } x_i  \in J $, for $ k \in \{ 1, \ldots, n \} $, form a Gr\"obner basis for $ J $ 
	with respect to any term order
	for which $ x_1', \ldots, x_n' $ are much larger than $ x_1, \ldots, x_n $. 	
\end{lemma}

{In particular, the dimension of the corresponding variety 
$ \Spec (\cL(\x,Q)) $ is $ n $ if $ Q $ is acyclic.}

\begin{Bem}
		In the skew-symmetrizable case, when $\Gamma(B)$ is acyclic, the exchange relations \eqref{exchangeskewsymmetrizable} generate the ideal of relations among the generators of $\cL(\x,B)$.
\end{Bem}

{A seed $ (\x,Q) $ is called {\em totally mutable} if it admits unlimited mutations in all directions. 
	Since we assume in this article that all vertices of a given quiver are mutable, 
	the  seeds $ ( \x,  Q ) $, which we consider, are always totally mutable.}

\begin{Thm}[{cf.~\cite[Theorem~1.20]{BFZ2005}}]
	\label{Thm:clus=low} 
	The cluster algebra $ \cA (\x, Q) $ associated with a totally mutable seed $ (\x, Q) $ is equal to the lower bound $ \cL(\x, Q) $ if and only if $ Q $ is acyclic.
\end{Thm}

\begin{example}
	\label{Ex:A_3}
	Let $ (\x,Q) $ be the seed corresponding to
	
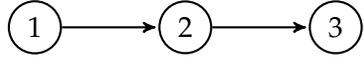
\begin{center} 
		\begin{tikzpicture}[->,>=stealth',shorten >=1pt,auto,node distance=2cm, thick,main node/.style={circle,draw}]
		
		\node[main node] (1) {$1$};
		\node[main node] (2) [right of=1] {$2$};
		\node[main node] (3) [right of=2] {$3$};
		
		\path
		(1) edge (2)
		(2) edge (3)
		;
		\end{tikzpicture}.
	\end{center}
	By the previous result, the cluster algebra $ \cA (\x,Q) $ is given by 
	\[
	\cA (\x,Q)
	= K [x_1, x_2, x_3, y_1, y_2 , y_3 ]
	/ I,
	\]
	\[
	I :=
	\langle 
	x_1 y_1 - x_2 - 1,
	\,
	x_2 y_2 - x_3 - x_1,
	\, 
	x_3 y_3 - 1 - x_2
	\rangle
	. 
	\]
\end{example}

Recall that a vertex $ i $ of a quiver $ Q $ is called 
a {\em sink} (resp.~{\em source})
if $ i $ is the target (resp.~source) of every arrow in $ Q $ incident to $ i $.
In Example~\ref{Ex:A_3}, the vertex $ 1 $ is a source, while $ 3 $ is a sink. 
As a consequence of {\cite[Proposition~9.2]{FZ2003II}}, one has the following lemma.

\begin{lemma}
	\label{Lem:tree_equiv}
	All orientations on a tree are mutation-equivalent via sequences of mutations at sinks and sources. 
\end{lemma}

\begin{example}
	Let us continue Example~\ref{Ex:A_3}.
	The exchange relations imply:
	\[
	\begin{array}{l}
	x_2 = x_1 y_1 - 1 = 
	\det \begin{pmatrix}
	x_1 & -1 \\
	-1 & y_1 
	\end{pmatrix};
	\\
	x_3  = x_2 y_2 - x_1 
	= x_1 y_1 y_2  - y_2 - x_1
	= 
	\det \begin{pmatrix}
	x_1 & -1 & 0 \\
	-1 & y_1 & -1 \\
	0 & -1 & y_2
	\end{pmatrix}.
	\end{array} 
	\]
	Therefore $ \Spec(\cA(\x,Q)) $ is isomorphic to a hypersurface,
	\[
	\cA := \cA(x,Q) 
	\cong K [x_1, y_1, y_2 , y_3] / \langle  
	x_1 y_1 y_2 y_3 - y_2 y_3 - x_1 y_3 - x_1 y_1 \rangle 
	.
	\]
	Using the Jacobian criterion, one determines that the singular locus of $ \Spec( \cA ) $ is the origin $  V (x_1, y_1 , y_2 , y_3 ) $. 
	Locally at the origin, $ 1 + y_2 y_3 $ is invertible. 
	In particular, we may introduce the local variable 
	$ z_1 := y_1 (1 + y_2 y_3 ) + y_3 $ and we obtain 
	\[
	x_1 y_1 y_2 y_3 - y_2 y_3 - x_1 y_3 - x_1 y_1 = 
	- (  y_2 y_3 + x_1 z_1 ) 
	\ . 
	\]
	Therefore, $ \Spec(\cA) $ has an singularity of type $ A_1 $ at the origin.
	In particular, the blowup of the origin resolves the singularities.  
\end{example}

The determinants arising above are examples of continuants. 
They will play a central role in our considerations, 
which is why we study some of their properties in the next section. 

\medskip

\subsection{Finite type classification}
\label{subsec:fin}

{The central object of the present paper are cluster algebras of finite type.
We end the section by recalling this notion as well as a classification theorem connecting cluster algebras of finite type with Dynkin \change{diagrams}.
Precise references for more details are \cite{FZ2003II}, \cite{FWZ2017}, or \cite[5.1]{Marsh13}, for example.}

\begin{defi}
	{Recall that we fixed a field $ \cF $, which is isomorphic to the field of rational functions in $ n $ variables over a field $ K $.}
	 
	\begin{enumerate}[leftmargin=*,label=(\arabic*)]
		\item 
		Let $ ( \x, B ) $ be a seed in $ \cF $. The cluster algebra $ \cA ( \x , B ) $ is said to be of {\em finite type} if 
		there are only finitely many distinct seeds mutation-equivalent to $ ( \x, B ) $. 
		
		\item  
		For any $n \times n$ square {integer} matrix $ B $,  
		its {\em Cartan counterpart} $ A (B) = (a_{i,j}) $, 
		is defined to be the integer matrix $ a_{i,i} := 2 $ and $ a_{i,j} := -|b_{i,j}| $ if $ i \neq j $. 
	
\end{enumerate}
\end{defi}

{Recall that a Cartan matrix $ A = (a_{i,j}) $ is called {\em of finite type} if all its principal minors are positive. 
	For the $ 2 \times 2 $ principal minors, this implies the condition $ a_{i,j} a_{j,i} \leq 3 $ for $ i \neq j $.}

\begin{defi}[{\cite[Definition~5.2.4]{FWZ2017}}] 
{Let $ A = (a_{i,j}) $ be an $ n \times n $ Cartan matrix of finite type.}
{The {\em Dynkin diagram} of $ A $ is 
		a graph with vertices $ \{ 1, \ldots, n \} $,
		for which the edges are	determined as follows:
		Let $ i, j \in \{ 1, \ldots, n \} $ with $ i \neq  j $. 
		If $ a_{i,j} a_{j,i} \leq 1 $,
		then the vertices $ i $ and $ j $ are joined by an edge if $ a_{i,j} \neq 0 $. 
		Whenever $ a_{i,j} a_{j,i} > 1 $, 
		the following rule is applied for the edge between $ i $ and $ j $:}
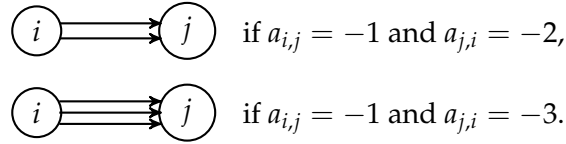
		\[
		\begin{array}{ll} 
			\begin{tikzpicture}[>=stealth',shorten >=1pt,auto,node distance=2cm, thick,main node/.style={circle,draw}]

			\node[main node] at (0,0) {$ i $};
			\node[main node] at (2,0) {$ j $};
			
			\draw[->] (0.3,0.1) -- (1.7,0.1);
			\draw[->] (0.3,-0.1) -- (1.7,-0.1);
			
			\end{tikzpicture}
			&
			\raisebox{7pt}{if $ a_{i,j} = - 1 $ and $ a_{j,i} = -2 $,}
			
			\\[5pt]
			
			\begin{tikzpicture}[->,>=stealth',shorten >=1pt,auto,node distance=2cm, thick,main node/.style={circle,draw}]

			\node[main node] at (0,0) {$ i $};
			\node[main node] at (2,0) {$ j $};
			
			\draw (0.3,0) -- (1.7,0);
			\draw (0.3,0.15) -- (1.7,0.15);
			\draw (0.3,-0.15) -- (1.7,-0.15);
			
			\end{tikzpicture}
			
			&
			\raisebox{7pt}{if $ a_{i,j} = - 1 $ and $ a_{j,i} = -3 $.}

		\end{array} 	
		\]
\end{defi}

{Here are two examples.
The graph on the right hand side is the Dynkin diagram of the corresponding matrix on the left hand side:}
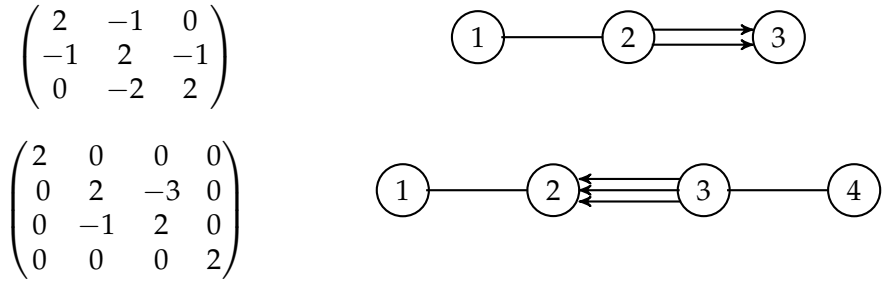
\[
	\begin{array}{ccc}
	
	\begin{pmatrix} 
	2 & -1 & 0 \\
	-1 & 2 & -1 \\
	0 & -2& 2 \\
	\end{pmatrix}  
	
	&
	
	\hspace{1cm}
	
	&
	
	\begin{tikzpicture}[>=stealth',shorten >=1pt,auto,node distance=2cm, thick,main node/.style={circle,draw}]

	\node[main node] at (0,0) {$ 1 $};
	\node[main node] at (2,0) {$ 2 $};
	\node[main node] at (4,0) {$ 3 $};
	
	\draw (0.31,0) -- (1.7,0);
	\draw[->] (2.31,0.1) -- (3.7,0.1);
	\draw[->] (2.31,-0.1) -- (3.7,-0.1);
	
	\end{tikzpicture}
	
	\\[25pt]

	\begin{pmatrix} 
	2 & 0 & 0 & 0 \\\
	0 & 2 & -3 & 0 \\
	0 & -1 & 2 & 0 \\
	0 & 0 & 0 & 2 \\
	\end{pmatrix}  
	
	&
	
	\hspace{1cm}
	
	&
	
	\begin{tikzpicture}[>=stealth',shorten >=1pt,auto,node distance=2cm, thick,main node/.style={circle,draw}]

	\node[main node] at (0,0) {$ 1 $};
	\node[main node] at (2,0) {$ 2 $};
	\node[main node] at (4,0) {$ 3 $};
	\node[main node] at (6,0) {$ 4 $};
	
	\draw (0.31,0) -- (1.7,0);
	\draw[->] (3.68,0) -- (2.3,0);
	\draw[->] (3.7,0.15) -- (2.3,0.15);
	\draw[->] (3.7,-0.15) -- (2.3,-0.15);
	\draw (4.31,0) -- (5.7,0);
	
	\end{tikzpicture}

	\end{array}
\]

There is the following classification of cluster algebras of finite type, cf.~\cite[Theorem~1.4]{FZ2003II}. 

\begin{Thm} \
	Let $ ( \x, B) $ be a seed.
	The cluster algebra $ \cA ( \x , B ) $ is of finite type 
	if and only if 
	the Cartan counterpart of one of its seeds is a Cartan matrix of finite type.
\end{Thm}

Recall, that the Cartan matrices $ A(B) $ of finite type are classified by the Dynkin diagrams 
$ A_n $,
$ B_n $,
$ C_n $,
$ D_n $, 
$ E_6, E_7, E_8, F_4, G_2 $
(for $ n \geq 1, 2, 3, 4 $ respectively),
see \cite[Theorem~5.2.6]{FWZ2017} or \cite[Section~6.4]{Carter2005}. 
\\

\begin{Bem} 
	{In the proof for the classification of finite type cluster algebras
	(see \cite{FZ2003} or \cite[Chapter 5]{FWZ2017}), 
	the non-quiver cases $ B_n, C_n, F_4, G_2 $ are connected to
	the quiver cases via the process of folding.
	The latter corresponds to taking a quotient with respect to a suitable group action on the quiver, see \cite[\S 4.4]{FWZ2017}. 
	More precisely, one has:
	\begin{itemize}
		\item 
		The seed pattern of type $ G_2 $ can be obtained from $D_4$ via folding 
		(\cite[\S~5.7]{FWZ2017}); 
		
		\item
		the seed pattern of type $ F_4 $ arises from $ E_6 $ through folding (\cite[Exercise 4.4.12 and \S~5.7]{FWZ2017});
		
		\item 
		the seed pattern of type $  C_n $ comes from $ A_{2n-1} $ via folding 
		(\cite[Proof of Theorem~5.5.2]{FWZ2017}); 
		
		\item
		we get the seed pattern of type $ B_n $ from $ D_{n+1} $ by folding 
		(\cite[Proof of Theorem~5.5.1]{FWZ2017}).  
		\\
	\end{itemize} 
}
\end{Bem}

\bigskip

\section{Continuant Polynomials} \label{S:continuants}

Continuants are classic in the study of determinants and were already considered by Euler.

\begin{defi} A \emph{continuant} of order $n$ is the determinant of a tri-diagonal matrix of the form 
$$\begin{pmatrix} y_1 & b_1 & 0 & \cdots & \cdots & \cdots \\
c_1 & y_2 & b_2 & 0 & \cdots & \cdots  \\
0 & c_2 & y_3 & b_3 & 0 & \cdots \\
\vdots & \ddots &  \ddots & \ddots &  \ddots & \ddots \\
0 & \cdots & 0  &  c_{n-2} & y_{n-1} & b_{n-1} \\
0 & \cdots & \cdots  & 0 & c_{n-1} & y_n \end{pmatrix} \ . $$
We will consider the special case where $b_i=c_i=-1$ for all $1 \leq i \leq n-1$, and denote this continuant by $P_n(y_1, \ldots, y_n)$. We set $P_0:=1$.
\end{defi}

These special continuants also appear in the work of Dupont \cite{Dupont} under the name \emph{generalized Chebyshev polynomials}. 

\begin{example} \label{Ex:continuants}
One can obtain all terms in the continuant from $y_1 \cdots y_n$ by replacing every pair of consecutive $y_i$ by $-1$ (see \cite[545]{Muir}). \\
For example,
one has $P_0=1$, $P_1(y_1)=y_1$, $P_2(y_1,y_2)=y_1y_2-1$ and 
$$P_3(y_1, y_2, y_3)= \det \begin{pmatrix} y_1 & -1 & 0 \\ -1 & y_2 & -1 \\ 0 & -1 & y_3  \end{pmatrix}=y_1y_2y_3-y_1-y_3 \ .$$

\end{example}

From the description as \change{a} determinant it is obvious that the continuant is symmetric: 
\begin{lemma}[Symmetry]\label{L:pnsym} 
	We have
$$P_n(y_1,\ldots,y_n)=P_{n}(y_n,\ldots,y_1).$$
\end{lemma}
The following properties are well-known, see e.g~\cite[Number 547 (3), Number 561,(4)]{Muir}:

\begin{lemma}[Recursion]\label{L:pnrecur} 
	The following recursion holds:
$$P_n(y_1,\ldots,y_n)=y_1P_{n-1}(y_2,\ldots,y_n)-P_{n-2}(y_3,\ldots,y_n).$$

Moreover, for $1 \leq r \leq n-1$ we have
$$P_n(y_1,\ldots,y_n)=P_{k}(y_1,\ldots,y_k)\cdot P_{n-k}(y_{k+1},\ldots,y_n)-P_{k-1}(y_1,\ldots,y_{k-1})\cdot P_{n-k-1}(y_{k+2},\ldots, y_n).$$
\end{lemma}

\begin{lemma}[Derivative] \label{L:pnder} 
	For $1 \leq k \leq n$, one has:
$$\frac{\partial }{\partial y_k}P_{n}(y_1,\ldots,y_n)= P_{k-1}(y_1,\ldots,y_{k-1})\cdot P_{n-k}(y_{k+1},\ldots,y_n).$$
\end{lemma}

From the description of the continuant of Example \ref{Ex:continuants} it is straightforward to verify that the terms of order $\leq 2$ of $P_n(y_1, \ldots, y_n)$, written $P_{n}(y_1, \ldots, y_n)_{\leq 2}$,  depend on $n \mod 4$ and are of the following form:
\begin{lemma}\label{pnquad}
	We have
\begin{align*}
P_{4k+1}(y_1,\ldots,y_{4k+1})_{\leq 2} & =y_1+y_3+\cdots+y_{4k+1} \ , \\
P_{4k+2}(y_1,\ldots,y_{4k+2})_{\leq 2} & =-1+y_1y_2+y_1y_4+\cdots+y_1y_{4k+2}+ y_3 y_4+ \cdots+y_{4k+1}y_{4k+2} \\
P_{4k+3}(y_1,\ldots,y_{4k+3})_{\leq 2} & =-y_1-y_3-\cdots-y_{4k+3} \\
P_{4k+4}(y_1,\ldots,y_{4k+4})_{\leq 2} & =1-y_1y_2-y_1y_4-\cdots-y_1y_{4k+4}- y_3 y_4-\cdots-y_{4k+3}y_{4k+4}
\end{align*}
\end{lemma}

As a preparation for the \change{remainder of the} article,
we study the singularities of the varieties determined by the continuants and deformations of them.

\begin{lemma}
	\phantomsection  
	\label{L:Sing(Pn)}
	\begin{enumerate}[leftmargin=*,label=(\arabic*)]
		\item 
		The variety $ \Spec(K[y_1, \ldots, y_n]/ \langle P_n \rangle) \subseteq \A^n_K$ is {regular} 
		for every $ n \in \ZZ_+ $,

		\item 
		The variety $ X_{2m+1,\lambda} := \Spec(K[y_1, \ldots, y_{2m+1}]/ \langle P_{2m+1} + \lambda \rangle) $ is {regular}
		for every $ m \in \ZZ_{\geq 0} $.

		\item We have 
			$ X_{2m+1,\lambda} \cong X_{2m+1,\mu} $ for every $ \lambda, \mu \in K \setminus \{ 0 \} $.
	\end{enumerate}
\end{lemma}

\begin{proof}
	(1)
	A straightforward induction using Lemma~\ref{L:pnrecur} shows that $ \langle P_n, \frac{\partial P_n}{\partial y_1} \rangle = \langle  1 \rangle $. 
	Thus, the Jacobian criterion implies the first claim.

	(2)
	The case $ \lambda = 0 $ follows from (1). 
		Hence, we assume $ \lambda \neq 0  $.
	Set $ n := 2m + 1 $ and $ Q_n := P_{n} + \lambda $.
	We compute the singular locus of $ Q_n $ using the Jacobian criterion.
	Notice that $ \frac{\partial Q_n}{\partial y_i} = \frac{\partial P_n}{\partial y_i} $ for all $ i \in \{ 1, \ldots, n \} $.  
	By the recursion of Lemma~\ref{L:pnrecur},  
	the vanishing of $ Q_n $ and 
	$\frac{\partial P_n}{\partial y_1} $ imply that we must have
	$
	P _{n-1} (y_2, \ldots, y_n) = 0
	$ 
	and  
	$	
	P_{n-2} (y_3, \ldots, y_n) = \lambda
	$.
	Computing the partial derivative $\frac{\partial P_n}{\partial y_2}$ (using Lemma~\ref{L:pnder})
	consequently yields $ y_1 = 0 $, 
	since $\lambda \neq 0$. 
	\\
	We prove by induction on $ k $ that we have
	\begin{enumerate} 
		\item[(a$_k$)] 
		$ y_1 = \cdots = y_{2k-1} = 0 $,
		
		\item[(b$_k$)] 
		$ P_{n-2k}(y_{2k+1}, \ldots, y_n) = (-1)^{k+1} \lambda $,
		and 
		
		\item[(c$_k$)] 
		$ P_{n-(2k-1)}(y_{2k}, \ldots, y_n) =  0 $.
	\end{enumerate}
	As we have just seen, all properties are true for $ k = 1 $. 
	Thus, let us discuss how to obtain (a$_{k+1}$), (b$_{k+1}$), (c$_{k+1}$) from (a$_{k}$), (b$_{k}$), (c$_{k}$).
	First, (a$_{k}$) implies $ P_{2k} (y_1, \ldots, y_{2k}) =  (-1)^{k}  $ by Lemma~\ref{pnquad}.
	\change{Therefore, we get (using Lemma~\ref{L:pnder})
	\[
	0 = \frac{\partial P_n}{\partial y_{2k+1}}
	= P_{2k}(y_1, \ldots, y_{2k}) \cdot P_{n-(2k+1)} (y_{2k+2}, \ldots, y_n)
	=  (-1)^k  P_{n-(2k+1)} (y_{2k+2}, \ldots, y_n)
	\ .
	\]
	This implies (c$_{k+1}$), i.e., $ P_{n-(2k+1)} (y_{2k+2},\ldots, y_n) =  0 $.} 
	\\
	Using (b$_k$) and Lemma~\ref{L:pnrecur}, we obtain 
	\[
	(-1)^{k+1} \lambda = P_{n-2k}(y_{2k+1}, \ldots, y_n) =
	- P_{n-(2k+2)} (y_{2k+3},\ldots, y_n)   
	\ ,
	\]
	or, in other words, (b$_{k+1} $) holds.
	\\
	\change{
	It remains to show (a$_{k+1}$). 
	Since (a$_k$) holds, we only have to prove $ y_{2k} = y_{2k+1} = 0 $. 
	The first recursion of Lemma~\ref{L:pnrecur} applied for $ P_{n-(2k-1)} (y_{2k},\ldots, y_n) $, (b$_{k}$), (c$_{k}$), and (c$_{k+1}$) provide
	$
	0 = P_{n-(2k+1)} (y_{2k+2},\ldots, y_n) = (-1)^{k+1} y_{2k} \cdot \lambda.
	$
	Since $ \lambda \neq  0 $,
	we get} 
	\[ 
		\change{y_{2k} = 0 .}
	\]
	Lemma~\ref{pnquad}, (a$_k$), and \change{$ y_{2k} = 0 $} lead to
	$ P_{2k+1} (y_1, \ldots, y_{2k+1}) = (-1)^k y_{2k+1} $
	and therefore, by Lemma~\ref{L:pnder}, 
	we have 
	\[
	0 = \frac{\partial P_n}{\partial y_{2k+2}}
	= P_{2k+1}(y_1, \ldots, y_{2k+1}) \cdot P_{n-(2k+2)} (y_{2k+3}, \ldots, y_n)
	=   \lambda y_{2k+1} 		
	\ .
	\]
	Since $ \lambda \neq 0 $, assertion (a$_{k+1}$) follows.	
	\\
	In particular, we get for (c$_{m+1}$): 
		$ 0 = P_{n-(2m+1)} = P_0  = 1 $,
		which is impossible
		and hence implies $ \Sing (X_{2m+1,\lambda}) = \varnothing $.

	(3)
	For the third part, it is sufficient to prove $ X_{2m+1,\lambda} \cong X_{2m+1,1} $ for every $ \lambda \neq 0 $.
	Since all terms appearing in $ P_{2m+1} $ are obtained from $ y_1 \cdots y_{2m+1} $ by replacing every pair of consecutive $ y_i $ by $ -1 $ (Example~\ref{Ex:continuants}),
	we have
	\[  
	P_{2m+1}( \lambda y_1, \lambda^{-1} y_2, \ldots,
	\lambda  y_{2i-1}, \lambda^{-1} y_{2i},  \ldots,  \lambda y_{2m+1}) 
	= 
	\lambda P_{2m+1}( y_1, y_2, y_3, \ldots, y_{2m}, y_{2m+1})
	\ , 
	\]
	which implies $ X_{2m+1,\lambda} \cong X_{2m+1,1} $.
\end{proof}

By Lemma~\ref{pnquad}, it is clear that $ P_{4k+2} + 1 $ and $ P_{4k} - 1 $ are singular at the origin. 
Hence, the analog of Lemma~\ref{L:Sing(Pn)}(2) is not true for $ n = 2m $ and singularities appear. 

\begin{proposition} 
	\label{Prop:SingPn}
	Let $ \lambda \in K \setminus \{ 0 \} $, $ m \in \ZZ_+ $,
	and
	$ X_{2m,\lambda} := \Spec(K[y_1, \ldots, y_{2m}]/ \langle P_{2m} + \lambda \rangle) $.
	We have
	\begin{equation}
	\label{Sing:Qn}
		\Sing ( X_{2m,\lambda})  =
		\begin{cases}
		V(y_1, \ldots, y_{2m}), & \mbox{if } \lambda = (-1)^{m+1},
		\\
		\varnothing, & \mbox{else.}
		\end{cases}
	\end{equation}
	In particular, $ X_{2m,\lambda} $ has at most an isolated singularity at the origin. 
	If $ \Sing(X_{2m,\lambda}) \neq \varnothing $, then
	$X_{2m,\lambda} $ has a singularity of type $ A_1 $ at the origin. 
	Therefore, blowing up the origin resolves the singularities of $X_{2m,\lambda} $. 
\end{proposition}

\begin{proof} 
	Set $ n := 2m $ and $ Q_n := P_{n} + \lambda $.
	We compute the singular locus of $ Q_n $ using the Jacobian criterion.
	As in the proof of Lemma~\ref{L:Sing(Pn)}(2), we
	prove by induction on $ k $ that we have
	\begin{enumerate} 
		\item[(a$_k$)] 
		$ y_1 = \cdots = y_{2k-1} = 0 $,
		
		\item[(b$_k$)] 
		$ P_{n-2k}(y_{2k+1}, \ldots, y_n) = (-1)^{k+1} \lambda $,
		and 
		
		\item[(c$_k$)] 
		$ P_{n-(2k-1)}(y_{2k}, \ldots, y_n) =  0 $.
	\end{enumerate}
	In particular, we get $ y_1 = \cdots = y_{n} = 0 $ by (a$_{m}$) and (c$_{m}$).
	In conclusion, we have 
	\[  
		\Sing(V(Q_n)) = V (Q_n, y_1, \ldots,  y_{n}) 
		= V ( (-1)^m + \lambda , y_1, \ldots,  y_{n})) 
		\ ,
	\] 
	which implies our claim on the singular locus, \eqref{Sing:Qn}. 
	
	Finally, let us classify the isolated singularity at the origin if $ \lambda = (-1)^{m+1} $. 
	We show that there is a coordinate transformation $ \change{( y_1, \ldots, y_{2m})} \mapsto (t_1, \ldots, t_{2m})$ 
	after localizing at the maximal ideal corresponding to the origin, such that 
	\[
		(-1)^{m+1}  Q_{2m}(t_1, \ldots, t_{2m}) = \sum_{i=1}^{m} t_{2i-1}t_{2i} 
		\ ,
	\]
	which implies that $ X_{2m,\lambda} $ has an $ A_1 $-singularity at the origin. 
	By Example \ref{Ex:continuants} we can write down all terms of the continuant explicitly. Note that $ Q_{2m}(y_1, \ldots, y_{2m}) $ only has terms of even order $\geq 2$:
	\begin{align*}
	(-1)^{m+1} Q_{2m}  = 
	& \, 
	y_1 y_2 + \cdots + y_1 y_{2m} + y_3 y_4 + \cdots + y_3 y_{2m} + y_5 y_{2m} + \cdots + y_{2m-1} y_{2m} - \\
	& \, 
	- y_1 y_2 y_3 y_4 - \cdots - y_1 y_2 y_{2m-1} y_{2m} - \cdots 
	- y_{2m-3} y_{2m-2} y_{2m-1} y_{2m} + \cdots + \\
	& \,
	+ (-1)^{m+1} y_1 y_2 \cdots y_{2m}
	\ .
	\end{align*}
	This can be written as 
	\begin{align*}
	(-1)^{m+1}  Q_{2m} = 
	& \,
	y_1 y_2
	+ \sum_{k=1}^{m-1} y_{2k+1} y_{2k+2} (-1)^k P_{2k}(y_1,\ldots, y_{2k}) +  \\
	& 
	\,
	+ y_1 y_4 + \cdots + y_1 y_{2m} + y_3 y_6 + \cdots + y_3 y_{2m} + \cdots + y_{2m-3} y_{2m} \change{=}
	\\
	\change{=}
	&
	\change{
		y_1 (y_2 + \sum_{\ell=2}^m y_{2 \ell} )
		+ 
		\sum_{k=1}^{m-1} y_{2k+1} \Big(  y_{2k+2} (-1)^k P_{2k}(y_1,\ldots, y_{2k})
		+ \sum_{\ell = k+2 }^m y_{2\ell} 
		\Big) 
	}
	\ . 
	\end{align*}
	By Lemma \ref{pnquad} the even continuants $P_{2k}$ yield locally around the origin units, so we may substitute for $k=0, \ldots, m-1$: 
	\change{$ t_{2k+2} := y_{2k+2} (-1)^k P_{2k}(y_1,\ldots, y_{2k})
		+ \sum_{\ell = k+2 }^m y_{2\ell} $.
	By further introducing $ t_{2m} := y_{2m} $ and $ t_{2k+1} := y_{2k+1} $ for $ k = 0, \ldots, m-1 $,
	we obtain that $ (-1)^{m+1}  Q_{2m} $ is of the desired form.}
\end{proof}

Observe that the statement and the proof of Proposition~\ref{Prop:SingPn} are independent of the characteristic $ p = \car(K) \geq 0 $ of the field. 
Nonetheless, the characteristic plays a role when it comes to the condition $ \lambda = (-1)^{m+1} $.
For example, if $ \lambda = 1 $ and $ m = 2k $, for some $ k \in \ZZ_+ $,
then
$ P_{4k} + 1 $ is {regular} if $ p \neq 2 $ since $ (-1)^{2k+1} = -1 \neq 1 $,
while it is singular if $ p = 2 $.

\bigskip

\section{Singularities of finite type cluster algebras coming from quivers}
\label{sec:Quiver}

\subsection{$\boldsymbol{A_n}$ cluster algebras}

Assume that $Q$ is a simply laced Dynkin diagram of type $A_n$ with any orientation. Since all trees with the same underlying undirected graph are mutation equivalent (Lemma~\ref{Lem:tree_equiv}), we may choose the following orientation:

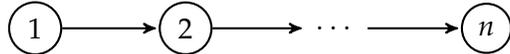
\begin{center} 
\begin{tikzpicture}[->,>=stealth',shorten >=1pt,auto,node distance=2cm, thick,main node/.style={circle,draw}]
\node[main node] (1) {$1$};
\node[main node] (2) [right of=1] {$2$};
\node (3) [right of=2] {$\cdots$};
\node[main node] (n) [right of=3] {$n $};
		
\path
(1) edge (2) 
(2) edge (3) 
(3) edge (n);
(2) edge (n);
\end{tikzpicture}
\end{center} 

Recall that we denote by $ \cA(A_n) $ the corresponding cluster algebra.

\begin{lemma} \label{L:presentationAn}
The cluster algebra $\cA(A_n) $ is isomorphic to $ K[z_1, \ldots, z_{n+1}]/\langle f_n \rangle $ with
 \[
\begin{array}{rl}  
f_n(z_1,\ldots,z_{n+1}) 
:= & 
\hspace{-6pt}
P_{n+1}(z_1, \ldots, z_{n+1})-1 \ .
\end{array} 
\]
Here, $P_{n+1}$ is the continuant polynomial defined in Section~\ref{S:continuants}. 
In particular, the variety $ \Spec(\cA(A_n)) $ is isomorphic to a hypersurface in $ \mathbb{A}_K^{n+1} $.
\end{lemma}

This result can also be found in \cite[Corollary~4.2]{Dupont}.
We provide a simpler and shorter proof. 

\begin{proof}
By Theorem~\ref{Thm:clus=low}, $\cA(A_n) $ is isomorphic to a quotient $K[x_1, \ldots, x_n,y_1, \ldots, y_n]/I$, where the ideal $I$ is generated by 
\[  
	x_1y_1-x_2-1,
	\ \  
	x_2y_2-x_1-x_3, 
	\ \  
	x_3y_3-x_2-x_4, 
	\ \ 
	\ldots,
	\ \ 
	x_{k}y_k-x_{k-1}-x_{k+1}, 
	\ \ 
	\ldots 
\]
\[ 
	\ldots, 
	\ \  
	x_{n-1}y_{n-1}-x_{n-2}-x_n, 
	\ \  
	x_ny_n-x_{n-1}-1 \ . 
\]
For $i \geq 2$ one can stepwise express each $x_i$ in terms of $x_1, y_1, \ldots, y_n$: The first equation shows that $x_2=x_1y_1-1=P_2(x_1, y_1)$. Substituting into the second equation yields
$$x_3=x_2y_2-x_1=P_2(x_1, y_1)y_2-P_1(x_1) \ , $$
which is by Lemma \ref{L:pnrecur} and Lemma \ref{L:pnsym} equal to $P_3(x_1, y_1, y_2)$. Recursively, we obtain for $k=2, \ldots, n$:
$$x_k=P_k(x_1, y_1, \ldots, y_{k-1}) \ . $$
Thus, the last generator $f_n:=x_ny_n-x_{n-1}-1$ becomes 
$$P_n(x_1, y_1, \ldots, y_{n-1})y_n - P_{n-1}(x_1, \ldots, y_{n-2})-1=P_{n+1}(x_1, y_1, \ldots, y_n)-1 \ .$$
In conclusion, we have  $ K[x_1, \ldots, x_n, y_1, \ldots, y_n]/I \cong K[z_1, \ldots, z_{n+1}]/\langle f_n \rangle $, 
where the generator on the right hand side is 
$f_n(z_1, \ldots, z_{n+1})=P_{n+1}(z_1, \ldots,z_{n+1})-1$. 
\end{proof}

\begin{Bem}
	Observe that the technique of the proof of Lemma~\ref{L:presentationAn} to reduce the number of generators using continuant polynomials can be applied for any quiver $ Q = ( Q_0, Q_1 ) $, 
	which contains a string of $ n $ vertices such that one them is a sink or a source.  
	More generally, if $ i \in Q_0 $ is a vertex of $ Q $ such that 
	$ \#\{ j \mid (i,j) \in Q_1 \} = 1 $ or $ \#\{ j \mid (j,i) \in Q_1 \} = 1 $,
	then the exchange relation at $ i $ is of the form
	$ x_i x_i' = x_k + \prod_{j \rightarrow i } x_j $
	or $ x_i x_i' = \prod_{j \leftarrow i } x_j  + x_k  $,
	for a unique vertex $ k \in Q_0 $,
	and hence $ x_k $ can be eliminated.

\end{Bem}

Lemma~\ref{L:presentationAn}, Lemma~\ref{L:Sing(Pn)}, and Proposition~\ref{Prop:SingPn} immediately imply Theorem~\ref{Thm:A} in the $ A_n $-case,
where it states:

\begin{corollary}
	\label{Cor:An} 
Let $\cA(A_n)$ be the cluster algebra of type $A_n$ over a field $ K $.
\\
If $\cchar(K) \neq 2$, then we have:
\begin{enumerate}[leftmargin=*,label=(\arabic*)]
\item $ \Spec(\cA(A_n)) $ is singular if and only if $n \equiv 3 \mod 4$. 
\item If $ n = 4m-1 $, for some $ m \in \ZZ_{> 0 } $, then $ \Spec(\cA(A_{4m-1})) $ is isomorphic to an isolated hypersurface singularity of type $A_1$ in $\A^{n+1}_K$. 
In particular, the resolution of singularities of $ \Spec(\cA(A_{4m-1})) $ is given by the blowup of the singular point.

\end{enumerate}
On the other hand, if $\cchar(K)=2$, then we have:  
\begin{enumerate}[leftmargin=*,label=(\arabic*')]
\item $ \Spec(\cA(A_n)) $ is singular if and only if $n \equiv 1 \mod 2$. 
\item If $n = 2m - 1 $, for some $ m \in \ZZ_{> 0 } $, then $ \Spec(\cA(A_{2m-1})) $ is isomorphic to an isolated hypersurface singularity of type $A_1$. 
In particular, the resolution of singularities of $ \Spec(\cA(A_{2m-1})) $ is given by the blowup of the singular point.
\end{enumerate}
\end{corollary}

\medskip

\subsection{$\boldsymbol{D_n}$ cluster algebras}

Next, we consider the quiver $ Q $, whose underlying graph is a simply laced Dynkin diagram of type $ D_n $, for some $ n \geq 4 $. 
Since all orientations on a tree are mutation equivalent (Lemma~\ref{Lem:tree_equiv}), 
we choose the following orientation and numbering of the vertices for $ Q $:
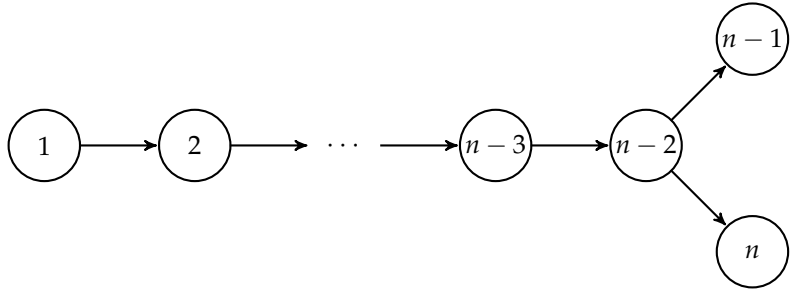
\begin{center} 
	\begin{tikzpicture}[->,>=stealth',shorten >=1pt,auto,node distance=2cm, thick,main node/.style={circle,draw,white}]
	
	\node[main node] (1) {1 1};	
	\node[main node] (2) [right of=1] {1 1};
	\node[main node] (3) [right of=2] {1 1}; 
	\node[main node] (4) [right of=3] {1 1}; 
	\node[main node] (5) [right of=4] {1 1};
	\node[main node] (6) [above right of=5] {1 1};
	\node[main node] (7) [below right of=5] {1 1};

	\node at (1) {\small $1$};
	\node at (2) {\small $2$};
	\node at (3) {\small $\cdots$}; 
	\node at (4) {\small $n-3$};
	\node at (5) {\small $n-2$};
	\node at (6) {\small $n-1$};
	\node at (7) {\small $n$};

	\draw (1) circle (13.5pt);
	\draw (2) circle (13.5pt);
	\draw (4) circle (13.5pt);
	\draw (5) circle (13.5pt);
	\draw (6) circle (13.5pt);
	\draw (7) circle (13.5pt);

	\path
	(1) edge (2) 
	(2) edge (3)
	(3) edge (4)
	(4) edge (5)
	(5) edge (7) 
	(5) edge (6);
	\end{tikzpicture}
\end{center} 
The corresponding cluster algebra,
which we denote by $ \cA(D_n) $, 
coincides with the lower cluster algebra of $ Q $ 
(Theorem~\ref{Thm:clus=low})
and the latter is completely described by its exchange relations
by Lemma~\ref{L:BFZ_Cor1.17}.

\begin{lemma}
	\label{Lem:pres_Dn}
	The cluster algebra
	$ \cA(D_n) $ 
	is isomorphic to 
	\[   
	K[u_1, u_2, u_3, u_4, z_1, \ldots, z_{n-2}]/ \langle h_1, h_2 \rangle ,
	\]
	where
	\[
	\begin{array}{ll}
	h_1 := 
	u_1 u_2 - u_3 u_4 - u_1 u_2 u_3 u_4 - u_2 u_4 P_{n-3} (z_1, \ldots, z_{n-3}),
	\\[5pt]
	h_2 := 
	u_3 u_4 - P_{n-2} (z_1, \ldots, z_{n-2}) - 1 .
	\end{array}
	\] 
	In particular, the variety $ \Spec(\cA(D_n)) $ is isomorphic to a subvariety of $ \A^{n+2}_K $ of codimension $ 2 $.
	(As before, $ P_{n-2} $ an $  P_{n-3}  $ are the continuant polynomials discussed in Section~\ref{S:continuants}.) 
\end{lemma}

\begin{proof}
	As mentioned before, we have
	$ 
	\cA(D_n)
	\cong  
	K [x_1, \ldots, x_n, y_1, \ldots, y_n ]/ I,
	$  
	where $ I $ is the ideal generated by 
	\begin{eqnarray}
	\label{eq:Dn_gone}
	x_1y_1-x_2-1, 
	&
	x_k y_k - x_{k+1} - x_{k-1}, \mbox{ for } k \in \{  2, \ldots, n-3\}
	\ ,  
	\end{eqnarray}
	\begin{equation}
	\label{eq:Dn_stay}
	g_3 := x_{n-2} y_{n-2} - x_{n-1} x_n - x_{n-3}, 
	\ \
	g_2 := x_{n-1} y_{n-1} - x_{n-2} - 1, 
	\ \
	g_1 := x_n y_n - x_{n-2} - 1 
	\ .
	\end{equation}
	As in the proof of Lemma~\ref{L:presentationAn}, we obtain from \eqref{eq:Dn_gone}
	$
	x_k = P_k (x_1, y_1, \ldots, y_{k-1}) $,
	for all $ k \in \{ 2, \ldots, n-2 \} $.
	The last generator in \eqref{eq:Dn_stay} can be replaced by 
	\[ 	
	g_1' := g_1 - g_2 = x_n y_n - x_{n-1} y_{n-1} 
	\]
	If we substitute $ x_2, \ldots, x_{n-2} $ in the remaining two generators, 
	we obtain
	\[
	g_2 
	= 
	x_{n-1} y_{n-1} - P_{n-2}(x_1, y_1, \ldots, y_{n-3}) -1
	\]
	and 
	\[
	\begin{array}{rcl}
	g_3  
	\hspace{-6pt} & = & \hspace{-6pt}
	P_{n-2}(x_1, y_1, \ldots, y_{n-3}) y_{n-2} - x_{n-1} x_n - P_{n-3}(x_1, y_1 \ldots, y_{n-4}) 
	= 
	\\[5pt]
	& = & \hspace{-6pt}
	y_{n-2} (x_{n-1} y_{n-1} - 1 - g_2 ) - P_{n-3}(x_1, y_1, \ldots, y_{n-4}) 
	- x_{n-1} x_n
	=
	\\[5pt]
	& =  & \hspace{-6pt}
	- x_{n-1} (  x_n - y_{n-2}  y_{n-1}) - y_{n-2} - P_{n-3}(x_1, y_1, \ldots, y_{n-4}) - y_{n-2} g_2 \
	. 
	\end{array} 	 	 
	\]
	We introduce 
	\[
	(u_1, u_2, u_3, u_4) := (x_n - y_{n-2} y_{n-1}, \, y_n, \, x_{n-1}, \, y_{n-1} ) 
	\]
	\[
	(z_1, z_2, \ldots, z_{n-2} ) :=  (x_1, y_1, \ldots, y_{n-3} )
	\]
	and obtain
	\[
	\begin{array}{l} 
	h_1 := \change{g_1'} = u_1 u_2 - u_3 u_4 +  y_{n-2} u_2 u_4  ,
	\\
	h_2 := g_2 = u_3 u_4 - P_{n-2} (z_1, \ldots, z_{n-2}) - 1  .
	\end{array}  
	\]
	On the other hand, $ g_3  + y_{n-2} g_2   \in I$ yields that we may eliminate 
	\[ 
	y_{n-2} = - u_1 u_3 - P_{n-3}(z_1, \ldots, z_{n-3}), 
	\]  
	which provides $ h_1 = 
	u_1 u_2 - u_3 u_4 - u_1 u_2 u_3 u_4 - u_2 u_4 P_{n-3} (z_1, \ldots, z_{n-3}) $,
	as desired. 
\end{proof}

By the previous result, $ \Spec ( \cA(D_n)) \cong V (h_1, h_2) \subset \A^{n+2}_K $. 
Using this presentation, we determine the singular locus of $ \Spec ( \cA(D_n)) $.

\begin{lemma}
	\label{Lem:SingDn}
	Let $\cA(D_n)$ be the cluster algebra of type $D_n$ over a field $ K $.
	\\
	If $\cchar(K) \neq 2$, then we have
	\[
	\Sing (\Spec ( \cA(D_n))) 
	\cong
	\begin{cases}
	Y_0 \cup Y_1 \cup Y_2 \cup Y_3 \cup Y_4
	& \text{ if } n \equiv 0  \mod 4 , \\
	Y_0 & \text{ otherwise} ,
	\end{cases} 
	\]
	where, for $ i \in \{ 1, 2, 3, 4 \} $, the component 
	$ Y_i$ is the $ u_i $-axis and so $ \dim(Y_i) = 1 $,
	while
	\[
	Y_0 := V (u_1, \ldots, u_4, \  P_{n-2} (z_1, \ldots, z_{n-2}) + 1) 
	\]
	and $ \dim (Y_0 ) = n - 3 $.
	The only possible singular point of $Y_0$ is the origin,
	\[ 
	\Sing(Y_0) 
	=
	\begin{cases}
	\bigcap\limits_{i=1}^4 Y_i 
	& \text{ if } n \equiv 0  \mod 4 \\
	\varnothing & \text{ otherwise} .
	\end{cases} 
	\]
	Observe that $ Y_0 $ has two irreducible components if $ n = 4 $ since $ P_2 (z_1,z_2) + 1 = z_1 z_2 $.
	\\
	On the other hand, if $ \cchar(K) = 2 $, the same statement holds true if we replace every condition $ n \equiv 0 \mod 4 $ by $ n \equiv 0 \mod 2 $. 
\end{lemma}

Note that Lemmas~\ref{Lem:pres_Dn} and~\ref{Lem:SingDn} for $ n = 4 $ imply Theorem~\ref{Thm:A}(4)(a).

\begin{proof}[Proof of Lemma~\ref{Lem:SingDn}]
	By Lemma~\ref{Lem:pres_Dn}, $ \Spec ( \cA(D_n)) $ is isomorphic to the subvariety of $ \A^{n+2}_K $ determined by 
	\[
	\begin{array}{l} 
	h_1 = u_1 u_2 - u_3 u_4 - u_1 u_2 u_3 u_4 - u_2 u_4 P_{n-3} (z_1, \ldots, z_{n-3})  = 0 ,
	\\
	h_2 = u_3 u_4 - P_{n-2} (z_1, \ldots, z_{n-2}) - 1  = 0 .
	\end{array}  
	\] 
	Observe that there is an ordering on the variables such that $ u_1 u_2 $ is the leading monomial of $ h_1 $
	and $ u_3 u_4 $ is the one of $ h_2 $.
	Hence, 
	\change{the dimension of $ \Spec ( \cA(D_n)) $ is equal to $ n $ and by applying the Jacobian criterion for smoothness,}
	the singular locus of $ \Spec ( \cA(D_n)) $ is determined by the vanishing of the $ 2 \times 2 $ minors of the Jacobian matrix of $ ( h_1, h_2 ) $.
	We abbreviate  
	\[  
	\Jac(D_n) := \Jac(h_1, h_2; u_1, u_2, u_3, u_4, z_1, \ldots, z_{n-2}). 
	\] 
	The first four columns of $ \Jac(D_n) $ are 
	\[
	\begin{pmatrix}
	u_2 (1-u_3 u_4) 
	& 
	u_1 (1-u_3 u_4) - u_4 P_{n-3}
	&
	-u_4 (1 + u_1 u_2)
	&
	-u_3 (1 + u_1 u_2 ) - u_2 P_{n-3} 
	\\
	0 & 0 & u_4 & u_3 
	\end{pmatrix}
	\]
	while the remaining columns are
	\begin{equation}
	\label{eq:secondmatrix}
	\begin{pmatrix}
	- u_2 u_4 \dfrac{\partial P_{n-3}}{\partial z_1}
	&
	\cdots  
	&
	- u_2 u_4 \dfrac{\partial P_{n-3}}{\partial z_{n-3}}
	& 
	0 
	\\[8pt]
	-\dfrac{\partial P_{n-2}}{\partial z_1}
	&
	\cdots 
	&
	
	-\dfrac{\partial P_{n-2}}{\partial z_{n-3}}
	&
	-\dfrac{\partial P_{n-2}}{\partial z_{n-2}}
	\end{pmatrix}
	.
	\end{equation} 
	(Here, we use the obvious abbreviations $ P_{n-2} $ and $ P_{n-3} $.)
	Clearly, the maximal minors of the first matrix are of the form $ u_3 ( \ldots ) $ and $ u_4 (\ldots ) $.
	Suppose $ u_3 = u_4 = 0 $. 
	The vanishing of $ h_1 $ and $ h_2 $ provides that for a singular point, 
	we have to have $ u_1 u_2 = 0 $ and $ P_{n-2} + 1 = 0 $.
	Further, the first row of the first matrix becomes $ \begin{pmatrix}
	u_2 & u_1 & 0 & - u_2 P_{n-3} 
	\end{pmatrix} $ 
	and every entry of
	the first row of the second matrix is zero.
	Thus, if $ u_1 = u_2 = 0 $, we obtain the irreducible component
	\[
	Y_0 = V (u_1, \ldots, u_4, P_{n-2} + 1 ) 
	\]
	in the singular locus. 
	On the other hand, suppose that $ u_1 \neq 0 $. 
	Since $ u_1 u_2 = 0 $ has to vanish, we get $ u_2 = 0 $. 
	The minors corresponding to the derivatives with respect to 
	$ ( u_2, z_i ) $, for $ i \in \{ 1, \ldots, n - 3 \} $ 
	provide that 
	\[
	\dfrac{\partial P_{n-2}}{\partial z_{1}} = \ldots = \dfrac{\partial P_{n-2}}{\partial z_{n-2}} = 0 .
	\]
	This yields the irreducible component
	$
	Y_1 = V (u_2, u_3, u_4 ) \cap \Sing (V(P_{n-2} + 1))
	$ 
	of the singular locus of $ \Spec( \cA(D_n) ) $.
	Analogously, we get 
	$
	Y_2 = V (u_1, u_3, u_4 ) \cap \Sing (V(P_{n-2} + 1))
	$ 
	if $ u_2 \neq 0 $. 
	
	Next, suppose that $ u_3 \neq 0 $ or $ u_4 \neq 0 $. 
	Then, the minors corresponding to the derivatives $ (u_1, u_4) $ and $ ( u_2, u_4 ) $
	(resp.~$ (u_1, u_3) $ and $ ( u_2, u_3 ) $)
	provide that we have to have
	\begin{equation}
	\label{eq:easy_der} 
	u_2 (1-u_3 u_4)  = 0 
	\ \ \ 
	\mbox{ and } 
	\ \ \ 
	u_1 (1-u_3 u_4) - u_4 P_{n-3} = 0 
	\end{equation} 
	for a singular point.
	If $ 1 - u_3 u_4 = 0 $, we get that $ P_{n-3}(z_1, \ldots, z_{n-3}) = 0 $. 
	On the other hand, the vanishing of $ h_2 $ yields $ P_{n-2}(z_1, \ldots, z_{n-2}) = 0 $,
	which is a contradiction as we have seen at the beginning of the proof of Lemma~\ref{L:Sing(Pn)}.
	\\
	Therefore, we get $ u_2 = 0 $ if $ u_3 \neq 0 $ or $ u_4 \neq 0 $. 
	This implies that all entries in the first row of the matrix \eqref{eq:secondmatrix} are zero. 
	Moreover, $ h_1 = h_2 = 0 $ is equivalent to $ u_3 u_4 =  P_{n-2} + 1 = 0 $.
	We have two cases:
	\begin{itemize}
		\item 
		$ u_3 = 0 $ and $ u_4 \neq 0 $.
		Then $ \dfrac{\partial h_1}{\partial u_3} = - u_4 \neq 0 $.  
		The minors corresponding to the derivatives with respect to $ (u_3, z_i) $ provide that all derivatives of $ P_{n-2} $ have to vanish.
		Since $ \dfrac{\partial P_{n-2}(z_1,\ldots, z_{n-2})}{ \partial z_{n-2}} = P_{n-3}(z_1,\ldots, z_{n-3}) $,
		the second equality of \eqref{eq:easy_der} and $ u_3 = 0 $ imply $ u_1 = 0 $. 
		Hence, we get the irreducible component 
		$
		Y_4 = V (u_1, u_2, u_3 ) \cap \Sing (V(P_{n-2} + 1))
		$ 
		
		\smallskip 
		
		\item 
		$ u_3 \neq 0 $ and $ u_4 = 0 $.
		Via the analogous arguments as in the previous case,
		we obtain the   
		irreducible component 
		$
		Y_3 = V (u_1, u_2, u_4 ) \cap \Sing (V(P_{n-2} + 1))
		$ 
		in the singular locus.
	\end{itemize} 
	Note that this covers all cases, where the minors of $ \Jac(D_n) $ vanish.
	Hence, we determined all components of the singular locus.   
	Furthermore, observe that
	\begin{equation}
	\bigcap_{i=1}^4 Y_i = \Sing (Y_0). 
	\end{equation}

	First, assume $ \cchar(K) > 2 $.
	By Lemma~\ref{L:Sing(Pn)} and Proposition~\ref{Prop:SingPn}, we have 
	\[
	\Sing( P_{n-2} (z_1, \ldots, z_{n-2} ) + 1 ) =
	\begin{cases} 
	V(z_1, \ldots, z_{n-2}) & \text{ if } n \equiv 0  \mod 4 \\
	\varnothing & \text{ otherwise} .
	\end{cases}
	\]
	This implies
	\[
	\Sing (\Spec ( \cA(D_n))) 
	\cong
	\begin{cases}
	Y_0 \cup Y_1 \cup Y_2 \cup Y_3 \cup Y_4
	& \text{ if } n \equiv 0  \mod 4 \\
	Y_0 & \text{ otherwise} ,
	\end{cases} 
	\]
	where $ \Sing(Y_0) = V (u_1, \ldots, u_4, z_1, \ldots, z_{n-2} )  = \bigcap_{i=1}^4 Y_i $  is the origin if $ n \equiv  0 \mod 4 $,
	while $ Y_0 $ is {regular} in the second case. 
	Observe that $ Y_i $ is the $ u_i $-axis and so $ \dim(Y_i) = 1 $, for $ i \in \{ 1, \ldots, 4 \} $, 
	and $ \dim (Y_0 ) = n - 3 $.	
	
	Let us turn to the case $ \cchar(K) = 2 $.
	The same arguments apply and the only difference appears, when we apply
	Proposition~\ref{Prop:SingPn}, which leads to the condition $ n \equiv 0 \mod 2 $ instead of $ n \equiv 0 \mod 4 $.
\end{proof}

As before in the $ A_n $-case, we can classify the singularities and construct a desingularization from this. 

\begin{prop}
	\label{Prop:TypeRSDn}
	Let $\cA(D_n)$ be the cluster algebra of type $D_n$ over a field $ K $.
	We use the notation of Lemma~\ref{Lem:SingDn}.
	\begin{enumerate}[leftmargin=*,label=(\arabic*)]
		\item 
		If $ \Sing (\Spec ( \cA(D_n))) \cong Y_0 $, 
		then the {variety} $ \Spec ( \cA(D_n)) $ is locally at $ Y_0 $ isomorphic to a cylinder over a hypersurface singularity of type $ A_1 $.
		In particular, the blowup with center $ Y_0 $ resolves the singularities. 
		
		\smallskip 
		
		\item 
		If $ \Sing (\Spec ( \cA(D_n))) \cong Y_0 \cup Y_1 \cup Y_2 \cup Y_3 \cup Y_4 $, then 
		we have:
		\begin{enumerate}[leftmargin=*]
			\item 
			For every $i \in \{0 , \ldots, 4 \} $,
			$ \Spec ( \cA(D_n)) $ is
			locally at a singular point different from the origin
			isomorphic to a cylinder over a hypersurface singularity of type $ A_1 $.
			
			\item 
			$ Y_0 $ is isomorphic to an $ ( n- 3 ) $-dimensional $ A_1 $-hypersurface singularity;
			
			\item 
			The singularities of $ \Spec ( \cA(D_n)) $ are resolved by first blowing up the origin and then choosing the strict transform of $ \bigcup_{i=0}^4 Y_i $ as the next center.
		\end{enumerate}	  
	\end{enumerate}
\end{prop}

\begin{proof}
	By Lemma~\ref{Lem:pres_Dn}, $ \Spec ( \cA(D_n)) $ is isomorphic to the subvariety of $ \A^{n+2}_K $ given by 
	\[
	\begin{array}{l} 
	h_1 = u_1 u_2 - u_3 u_4 - u_1 u_2 u_3 u_4 - u_2 u_4 P_{n-3} (z_1, \ldots, z_{n-3})  = 0 
	\ ,
	\\
	h_2 = u_3 u_4 - P_{n-2} (z_1, \ldots, z_{n-2}) - 1  = 0 
	\ .
	\end{array}  
	\]
	First, suppose $ \Sing (\Spec ( \cA(D_n))) \cong Y_0 $, where 
	\[  
	Y_0 = V ( u_1, \ldots, u_4, \  P_{n-2} (z_1, \ldots, z_{n-2}) + 1)
	\ ,
	\] 
	is {regular} and of dimension $ n - 3 $. 
	Moreover, Lemma~\ref{L:Sing(Pn)} and Proposition~\ref{Prop:SingPn} provide that $ H := \Spec(K[u_1, \ldots,u_4, z_1, \ldots, z_{n-2}]/\langle h_2 \rangle ) $ is {regular}.
	Locally at $ Y_0 $, the element $ 1 - u_3 u_4 $ is invertible
	and thus we may introduce the local variable $ w_{1} := u_1 (1 - u_3 u_4) - u_4 P_{n-3}  $.
	Using the latter, we get $ h_1 = w_1 u_2 - u_3 u_4 $ locally.
	Therefore, locally at $ Y_0 $,
	the variety $ \Spec(\cA(D_n)) $ is isomorphic to an intersection of a cylinder over an $ A_1 $-hypersurface singularity and a {regular} variety $ H $, which is transversal to the cylinder.
	In particular, the blowup of $ Y_0 $ is a desingularization of $ \Spec(\cA(D_n)) $.
	This ends the proof of part (1).

	Let us come to the case $ \Sing (\Spec ( \cA(D_n))) \cong Y_0 \cup Y_1 \cup Y_2 \cup Y_3 \cup Y_4 $.
	By Lemma~\ref{Lem:SingDn}, this can only happen if $ n \equiv 0 \mod 2 $.
	(Note that $ n \equiv 0 \mod 4 $ implies $ n \equiv 0 \mod 2 $.) 
	Here, the singular locus of $ \Spec(\cA(D_n)) $ has five components, 
	$ Y_0 $ above and the $ u_i $-axes $ Y_i $ for $ i \in \{ 1, 2, 3, 4 \} $.  
	The only singular point of $ \Sing (\Spec(\cA(D_n))) $ is the origin $ 0 $,
	which is also the singular locus of $ Y_0 $, 
	as well as the intersection of the four other components $ Y_1, \ldots, Y_4 $. 
	\\
	The same argument as above shows that, 
	locally at a singular point, which is contained in $ Y_0 \setminus \{0\} $, the {variety} $ \Spec(\cA(D_n)) $ is isomorphic to a cylinder over an $ A_1 $-singularity intersected with a {regular} hypersurface, which is transversal to the cylinder.
	
	Let us consider the other components. 
	In the proof of Proposition~\ref{Prop:SingPn}, we have seen that there is 
	a  local coordinate transformation $(z_1, \ldots, z_{n-2}) \mapsto (t_1, \ldots, t_{n-2})$, such that 
	\begin{equation}
	\label{eq:hom} 
	P_{n-2}(t_1, \ldots, t_{n-2}) +1  =\sum_{i=1}^{m}t_{2i-1}t_{2i} \ , 
	\end{equation} 
	where $ m := \frac{n-2}{2} \in \ZZ_+ $, which is an integer since $ n \equiv 0 \mod 2 $.
	
	Along the $ u_1 $-axis without the origin, $ Y_1 \setminus \{ 0 \} $,
	the term $ u_1( 1 - u_3 u_4) - u_4 P_{n-3} $  is  invertible and hence 
	we may introduce $ w_2 := h_1 = u_2 ( u_1 ( 1 - u_3 u_4) - u_4 P_{n-3}) - u_3 u_4 $ to replace $ u_2 $. 
	Thus, locally at a point of $ Y_1 \setminus \{ 0 \} $, we get that $ \Spec(\cA(D_n)) $ is isomorphic to the hypersurface
	\[
	\Spec( K [u_1, u_3, u_4, t_1, \ldots, t_{n-2}]/ \langle u_3 u_4 - \sum_{i=1}^{m}t_{2i-1}t_{2i} \rangle )
	\ ,
	\]
	which is a cylinder over an $ A_1 $-hypersurface singularity. 
	The analogous situation appears for $ Y_2 \setminus \{ 0 \} $.
	\\
	Let us consider the local situation at $ Y_3 \setminus \{ 0 \} $. 
	There, $ u_3 $ is invertible so that it makes sense to define 
	$ w_1 := u_1 u_3, w_2 := u_2 u_3^{-1}, w_4 :=  u_3 u_4 $.
	Using this, we obtain
	\[ 
	h_1 = u_1 u_2 - u_3 u_4 - u_1 u_2 u_3 u_4 - u_2 u_4 P_{n-3} = w_2 (w_1 (  1 - w_4) - w_4 P_{n-3}) - w_4 
	\ .
	\]
	Furthermore, we may introduce $ v_1 := w_1 (  1 - w_4) - w_4 P_{n-3} $ 
	so that 
	the vanishing of $ h_1 $ allows to eliminate the variable $ w_4 $ in $ h_2 = w_4 - P_{n-2} - 1 $.
	Hence, 
	locally at a point of $ Y_3 \setminus \{ 0 \} $, 
	the {variety} $ \Spec(\cA(D_n)) $ is isomorphic to
	\[
	\Spec( K [v_1, w_2, u_3, t_1, \ldots, t_{n-2}]/ \langle v_1 w_2 - \sum_{i=1}^{m}t_{2i-1}t_{2i} \rangle )
	\ .
	\]
	In other words, we are in the same situation as for $ Y_1 \setminus \{ 0 \} $.
	The analogous argument (and using Lemma~\ref{pnquad}) provide the same result for $ Y_4 \setminus \{ 0 \} $.
	Hence, we have shown $ (2)(a) $.
	
	It remains to study the situation at the origin, which is the singular locus of $ Y_0 $ 
	and also equal to $ \bigcap_{i=1}^4 Y_i $. 
	By \eqref{eq:hom}, $ Y_0 $ is isomorphic to a hypersurface singularity in $ \mathbb{A}_K^{n-2}  $ 
	of type $ A_1 $. 
	In particular, we get $ (2)(b) $
	and blowing up the origin resolves the singularities of $ Y_0 $. 
	\\
	Finally, for $ (2)(c) $, the same argument as above (for $ Y_0 $) provides that
	$ h_1 = w_1 u_2 - u_3 u_4 $ locally at the origin. 
	In particular, $ h_1 $ and $ h_2 $ are both homogeneous of degree $ 2 $. 
	This implies, if we blow up the origin, 
	then the singular locus of the strict transform of $ \Spec( \cA(D_n)) $ is equal to $ Y_0'\cup Y_1'\cup \ldots \cup Y_4' $,
	where $ Y_i' $ denotes the strict transform of $ Y_i $.
	Furthermore, for every $ i \neq j $, we have $ Y_i' \cap Y_j' = \varnothing $. 
	Therefore $ Z:= \bigcup_{i=0}^4 Y_i' $ is regular and after blowing up with center $ Z $ all singularities are resolved by $ (2)(a) $.
\end{proof}

\medskip

\subsection{$\boldsymbol{E_6, E_7, E_8} $ cluster algebras}

Let us now turn our attention to the missing \change{skew-symmetric} cluster algebras of finite type,
which are those arising from orientations on  $ E_6, E_7, E_8 $ Dynkin diagrams. 
As before, we fix a field $ K $. 

\begin{proposition}
	\label{Prop:E6_8}
	Let $\cA(E_6)$ (resp.~$\cA(E_8)$) be the cluster algebra of type $ E_6 $ (resp.~$ E_8 $) over $ K $.
	\begin{enumerate}
		\item 
		There exist presentations of $ \cA(E_6) $ and $\cA(E_8)$ of codimension three.
		
		\item
		The {varieties} $ \Spec ( \cA(E_6)) $ and $  \Spec ( \cA(E_8)) $ are {regular}. 
	\end{enumerate}
\end{proposition}

\begin{proof}
	We begin with $ E_8 $.	
	By Lemma~\ref{Lem:tree_equiv}, we may choose any orientation for the quiver, whose underlying graph is the $ E_8 $ Dynkin diagram.
	We choose: 
	
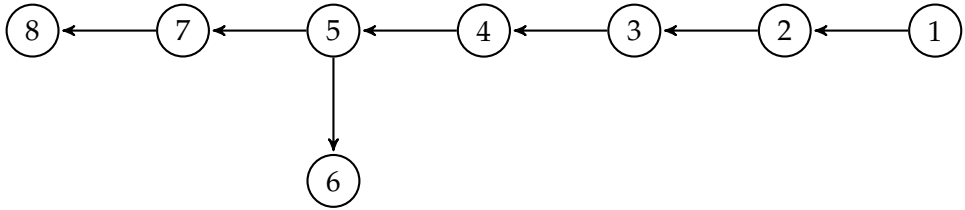
\begin{center} 
		\begin{tikzpicture}[->,>=stealth',shorten >=1pt,auto,node distance=2cm, thick,main node/.style={circle,draw}]

		\node[main node] (1) {$ 1 $};	
		\node[main node] (2) [left of=1] {$ 2 $};
		\node[main node] (3) [left of=2] {$ 3 $};
		\node[main node] (4) [left of=3] {$ 4 $};
		\node[main node] (5) [left of=4] {$ 5 $};
		\node[main node] (6) [below of=5] {$ 6 $};
		\node[main node] (7) [left of=5] {$ 7 $}; 
		\node[main node] (8) [left of=7] {$ 8 $};

		\path
		(1) edge (2) 
		(2) edge (3)
		(3) edge (4)
		(4) edge (5)
		(5) edge (7)
		(5) edge (6)
		(7) edge (8);
		\end{tikzpicture}
	\end{center} 
	Using the analogous arguments as as in the $D_n$-case, we obtain
	\[
	\cA(E_8) \cong K[x_1, x_6, x_8, y_1, \ldots, y_8]
	/
	\langle h_1, h_2 , h_3 \rangle ,
	\]
	where we define
	\[
	\begin{array}{ll}
	h_1 := 
	P_5 (x_1, y_1, y_2, y_3, y_4) -  P_2 (x_6, y_6) ,
	\\[5pt] 
	h_2 := 
	P_2 (x_6, y_6)  y_5 - x_6 P_2 (x_8, y_8) - P_4(x_1, y_1, y_2, y_3), 
	\\[5pt]
	h_3 := 
	P_3 (x_8,y_8,y_7) - P_2(x_6, y_6) .
	\end{array} 
	\]
	The singular locus of $ \Spec(\cA(E_8)) $ is determined by the vanishing locus of the $ 3 \times 3 $ minors of the Jacobian matrix
	\[   
	\Jac(E_8) := \Jac(h_1, h_2, h_3; x_1, x_6, x_8, y_1, \ldots, y_8) .
	\] 
	First, observe that $ h_3 = P_3 (x_8,y_8,y_7) - P_2(x_6,y_6) 
	= x_8 y_8 y_7 - x_8 - y_7  - x_6 y_6 + 1 $. 
	\change{We get}
	\begin{equation}
	\label{eq:der_p_3}
	\dfrac{\partial h_3}{\partial x_8} = y_7 y_8 - 1 
	\neq  0, 
	\hspace{10pt}
	\mbox{ or } 
	\hspace{10pt}
	\dfrac{\partial h_3}{\partial y_7} = x_8 y_8 - 1
	\neq  0, 
	\hspace{10pt}
	\mbox{ or } 
	\hspace{10pt}
	\dfrac{\partial h_3}{\partial y_8} = x_8 y_7 
	\neq  0,
	\end{equation}
	\change{where the non-vanishing of at least one derivative can be seen by setting two derivatives equal to zero which leads to the third derivative being non-zero.}
	We fix 	$ z \in \{ x_8, y_7, y_8 \} $ such that $ \frac{\partial h_3}{\partial z} \neq 0 $.
	The minor determined by the derivatives with respect to 
	$ ( y_5, y_6, z ) $ provides
	that we must have $ x_6 (x_6 y_6 - 1 ) =  0 $ at a singular point of $ \Spec(\cA(E_8)) $.
	If
	$ x_6 y_6 - 1 = 0 $,
	then vanishing of $ h_3 $ and of the minor of $ \Jac(E_8) $ determined by the derivatives with respect $ ( x_6, y_6, z) $ 
	lead to a contradiction.
	Thus, we get $ x_6 = 0 $.
	The columns of $ \Jac(E_8) $ corresponding to $ z $ and $ y_5 $ become the transpose of the vectors 
	$ (0, 0,  \frac{\partial h_3}{\partial z} ) $
	and $ (0,1, \ast) $, for some entry $ \ast $.
	Since $ \frac{\partial h_3}{\partial z} $ is non-zero,
	we can only have a singular point
	if $ \Jac(h_1; x_1, x_6, y_1, \ldots, y_4, y_6 ) $ is the zero vector.
	Note that 
	$ \frac{\partial h_1}{\partial x_6} = - y_6 $
	and $ \frac{\partial h_1}{\partial y_6} = - x_6 = 0 $.
	Using $ x_6 = 0 $, the vanishing of $ h_1 $ can be reformulated as 
	\[
	Q_5 :=  P_5 (x_1, y_1, y_2, y_3, y_4) + 1 = 0. 
	\]
	Since 
	$ 
	\frac{\partial h_1}{\partial x_1}
	= \frac{\partial Q_5 }{\partial x_1}
	$
	and
	$
	\frac{\partial h_1}{\partial y_i}
	= \frac{\partial Q_5 }{\partial y_i},
	$ 
	for 
	$
	i \in \{ 1, \ldots, 4 \},
	$ 
	we get the inclusion 
	\[ 
	\Sing(\Spec ( \cA(E_8))) \subseteq \Sing (V(Q_5)) 
	=
	\varnothing 
	\ ,
	\]
	where the last equality holds by Lemma~\ref{L:Sing(Pn)}.	
	This concludes the proof for $ E_8 $.

	The statement for $ E_6 $ follows by applying the analogous arguments for the 
 quiver:
	
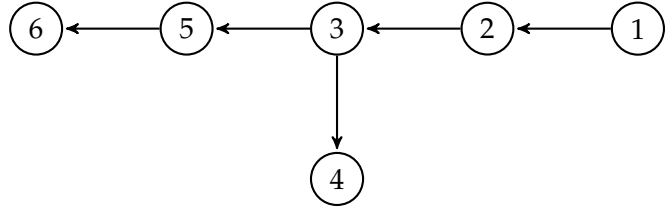
\begin{center} 
		\begin{tikzpicture}[->,>=stealth',shorten >=1pt,auto,node distance=2cm, thick,main node/.style={circle,draw}]

		\node[main node] (3) {$ 1 $};
		\node[main node] (4) [left of=3] {$ 2 $};
		\node[main node] (5) [left of=4] {$ 3 $};
		\node[main node] (6) [below of=5] {$ 4 $};
		\node[main node] (7) [left of=5] {$ 5 $}; 
		\node[main node] (8) [left of=7] {$ 6 $};

		\path
		(3) edge (4)
		(4) edge (5)
		(5) edge (7)
		(5) edge (6)
		(7) edge (8);
		\end{tikzpicture}
	\end{center}
	We leave the details as an easy exercise for the reader.   
\end{proof}

\begin{proposition}
	Consider the variety $ \Spec ( \cA(E_7)) $ over any field $ K $
	corresponding to the cluster algebra of type $ E_7 $. 
	\begin{enumerate}[leftmargin=*,label=(\arabic*)]
		\item 
		If $ \car(K) \neq 2 $, then $ \Spec ( \cA(E_7)) $ is {regular}.
		
		\item 
		If $ \car(K) = 2 $, then $ \Spec ( \cA(E_7)) $ is isomorphic to $ 7 $-dimensional subvariety of $ \A^{10}_K $, whose singular locus is a {regular} surface. 
		Locally at the singular locus, $ \Spec ( \cA(E_7)) $ is isomorphic to 
		a 
		cylinder over an isolated hypersurface singularity of type $ A_1 $ in $ \A_K^6 $ intersected with a {regular} hypersurface, which is transversal to the cylinder.
		In particular, the resolution of singularities of $ \Spec ( \cA(E_7)) $ is given by the blowup of the singular locus.	
	\end{enumerate}
\end{proposition}

\begin{proof}
	Analogous to the proof of Proposition~\ref{Prop:E6_8},
	we choose the quiver
	
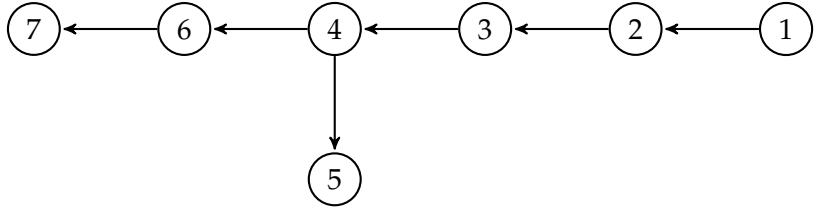
\begin{center} 
		\begin{tikzpicture}[->,>=stealth',shorten >=1pt,auto,node distance=2cm, thick,main node/.style={circle,draw}]

		\node[main node] (2) {$ 1 $};
		\node[main node] (3) [left of=2] {$ 2 $};
		\node[main node] (4) [left of=3] {$ 3 $};
		\node[main node] (5) [left of=4] {$ 4 $};
		\node[main node] (6) [below of=5] {$ 5 $};
		\node[main node] (7) [left of=5] {$ 6 $}; 
		\node[main node] (8) [left of=7] {$ 7 $};

		\path
		(2) edge (3)
		(3) edge (4)
		(4) edge (5)
		(5) edge (7)
		(5) edge (6)
		(7) edge (8);
		\end{tikzpicture}
	\end{center} 
	and we obtain
	$ 
	\cA(E_7) \cong K[x_1, x_5, x_7, y_1, \ldots, y_7]
	/
	\langle h_1, h_2 , h_3 \rangle ,
	$ 
	where we define
	\[
	\begin{array}{ll}
	h_1 := 
	P_4(x_1, y_1, y_2, y_3) -  P_2 (x_5, y_5),
	\\[5pt] 
	h_2 := 
	P_2 (x_5, y_5)  y_4 - x_5 P_2 (x_7, y_7) - P_3(x_1, y_1, y_2), 
	\\[5pt]
	h_3 := 
	P_3 (x_7,y_7,y_6) - P_2(x_5, y_5) .
	\end{array} 
	\]
	The same arguments as in the $ E_8 $ case provide that
	$	\Sing (\Spec ( \cA(E_7))) $ is isomorphic to a subvariety of $ V (x_5) \cap \Sing (V(Q_4)) $, 
	where 
	$  
	Q_4 := Q_4 (z_1,z_2,z_3,z_4):= P_4(z_1,z_2,z_3,z_4)  + 1 
	$
	and 
	$ (z_1, z_2, z_3, z_4) := (x_1, y_1, y_2, y_3 ) $.
	Since $ \car(K) \neq 2 $, we have $ \Sing (V(Q_4)) = \varnothing $,
	by Proposition~\ref{Prop:SingPn},
	which concludes the the proof of (1).

	(2) 
	Suppose $ \car(K) = 2 $. 
	We choose  $ z \in \{ x_7, y_6, y_7 \} $ such that $ \frac{\partial h_3}{\partial z} \neq 0 $.
	Then the minor of $ \Jac(E_7)  := \Jac(h_1, h_2, h_3; x_1, x_5, x_7, y_1, \ldots, y_7) $ corresponding to the derivatives with respect to 
	$ (y_4, y_4,z) $ and $ (x_5, y_5, z) $ provide that at a singular point, we have $ x_5 = 0 $. 
	We get that $ P_2 (x_5, y_5) = 1 $
	and hence 
	\[  
	h_1 = 
	P_4(x_1, y_1, y_2, y_3) -  P_2 (x_5, y_5)
	=
	f_3 (x_1, y_1, y_2, y_3) .
	\] 
	The minors corresponding to $ (*, y_4, z) $, where $ * \in \{ x_1, x_5, y_1, y_2, y_3 \} $ lead to the equality
	\[
	\Sing (\Spec ( \cA(E_7))) 
	= \Sing (V(f_3)) \cap V(x_5, h_2, h_3)
	= V (x_1, x_5, y_1, \ldots, y_5, P_3 (x_7, y_7,y_6)+1) 
	\ , 
	\] 
	which is {regular} 
	by 
	Lemma~\ref{L:Sing(Pn)}.
	\\
	Let us consider the situation locally at the singular locus. 
	Then, $ P_2 (x_5,y_5) $ and $ P_2 (y_2,y_3) = 1 + y_2 y_3 $ are units. 
	Thus, we may introduce the local coordinates
	$ z_1 := y_1 (1 + y_2 y_3) + y_3 $
	and 
	$ z_4 := h_2   = 
	P_2 (x_5, y_5)  y_4 - x_5 P_2 (x_7, y_7) - P_3(x_1, y_1, y_2) $.
	As in \eqref{eq:der_p_3}, the derivatives of 
	$ h_3 = 
	P_3 (x_7,y_7,y_6) - P_2(x_5, y_5) $ 
	with respect to $ x_7$, $ y_7 $, and $ x_6 $ cannot vanish at the same time, 
	i.e., $ V (h_3 ) $ is a {regular} hypersurface, which is transversal to $ V( z_4 ) $.
	Observe that (using $ \cchar(K)=2 $)
	\[ 
	\begin{array}{l} 
	h_1 = 
	x_1 y_1 y_2 y_3 + x_1 y_1 + x_1 y_3 + y_2 y_3 + x_5 y_5  
	= 
	x_1 z_1 + x_5 y_5 + y_2 y_3 \ .
	\end{array} 
	\]
	This implies the remaining parts of the proposition. 
\end{proof}

\bigskip

\section{Singularities of finite type cluster algebras not coming from quivers}

\label{sec:Matrix}

{Next, let us discuss the singularities of cluster algebras of finite type for which it is necessary to work with skew-symmetrizable matrices.
Recall the exchange relations~\eqref{exchangeskewsymmetrizable} in the matrix setting (Remark~\ref{Bem:Mat}) 
as well as the definitions of subsection~\ref{subsec:fin}.}

\medskip

\subsection{$\boldsymbol{B_n}$ cluster algebras}

A possible exchange matrix $ B $ for type $B_n$, {$ n \geq 2 $}, is given as
\[ 
	B=\begin{pmatrix} 0 & 1 & \cdots & 0 & 0 & 0 & 0 \\\
	-1 & 0 & \ddots & 0 & 0 & 0 & 0 \\
	\vdots & \ddots & \ddots & \ddots & \vdots & \vdots & \vdots  \\
	0 & 0 & \ddots & 0 &1 & 0 & 0 \\
	0 & 0  & \cdots & -1 & 0 &1 & 0 \\
	0 & 0 & \cdots & 0 & -1 & 0 & 1 \\
	0 & 0 & \cdots & 0   &0 &-2  & 0 \\
\end{pmatrix}  
\]
(cf.~\cite[\S 5.5, (5.31)]{FWZ2017})
and {the corresponding Dynkin diagram is of type $ B_n $}
(where $ n $ is the number of vertices); 
\[
\begin{tikzpicture}[>=stealth',shorten >=1pt,auto,node distance=2cm, thick,main node/.style={circle,draw}]

\node[main node] at (0,0) { \ \ \ };
\node[main node] at (2,0) { \ \ \ };
\node at (4,0) {$ \cdots $};
\node[main node] at (6,0) { \ \ \ };
\node[main node] at (8,0) { \ \ \ };

\draw (0.25,0) -- (1.75,0);
\draw (2.25,0) -- (3.5,0);
\draw (4.5,0) -- (5.75,0);
\draw[->] (6.25,0.1) -- (7.75,0.1);
\draw[->] (6.25,-0.1) -- (7.75,-0.1);

\end{tikzpicture}
\ .
\]

\begin{lemma}
	\label{L:Bn_Pres}
	Let $ K $ be any field and {$ n \geq 2 $}.
	The cluster algebra 
	$ \cA(B_n) $ is isomorphic to 
	\[ 
	{K[z_1, \ldots, z_{n-1},u_1,u_2, u_3]/\langle g_n, h_n \rangle}
	\ , 
	\]
	where 
	{$ g_n := (u_1 u_2 - 1 ) u_3 - u_1^2 - P_{n-2}(z_1,\ldots, z_{n-2})  $
	and 
	$ h_n := u_1 u_2 - 1 - P_{n-1}(z_1,\ldots, z_{n-1}) $. 
	In particular, the variety $ \Spec(\cA(B_n)) $ is isomorphic to a codimension two subvariety of $ \mathbb{A}_K^{n+2} $.}
\end{lemma}

\begin{proof}
Since the underlying diagram $\Gamma(B) $ is acyclic of type $B_n$, we get the presentation
\begin{align*}
\mathcal{A}(B_n) \cong K[x_1, \ldots, x_n, y_1, \ldots, y_n]/ & \langle x_1y_1-x_2-1, x_2y_2 - x_1 - x_3, \ldots, x_{n-2}y_{n-2}-x_{n-3}-x_{n-1}, \\ & 
{x_{n-1}y_{n-1} - x_n^2 - x_{n-2}}, x_n y_n -x_{n-1}-1 \rangle \ . 
\end{align*}
As in {the proof of Lemma~\ref{L:presentationAn} (type $ A_n $)}, 
one can express $x_k$ in terms of $x_1, y_1 ,\ldots, y_{k-1}$:
\[ 
	x_k=P_{k}(x_1, \ldots, y_{k-1}),  
	\ \ \mbox{ for } 2 \leq k \leq n-1 \ .
\]

{If we plug this into the remaining generators $x_{n-1}y_{n-1} - x_n^2 - x_{n-2} $ and $  x_n y_n -x_{n-1}-1 $, 
we obtain
\[	 
	\begin{array}{lll} 
	\widetilde g_n := P_{n-1}(x_1,y_1,\ldots, y_{n-2}) y_{n-1} - x_n^2 - P_{n-2}(x_1, y_1, \ldots, y_{n-3}) = 0 \ , 
	\\[5pt]
	h_n := x_n y_n - P_{n-1}(x_1,y_1,\ldots, y_{n-2}) - 1 = 0  \ .
	\end{array} 
\] 
We replace $ \widetilde g_n $ by $ g_n := \widetilde g_n + y_{n-1} h_n = 
(x_n y_n - 1 ) y_{n-1} - x_n^2 - P_{n-2}(x_1,y_1,\ldots, y_{n-3}) $. 
Therefore, 
$\mathcal{A}(B_n)$ is isomorphic to $K[x_1, x_n, y_1, \ldots, y_n]/\langle g_n , h_n\rangle$,
which yields the assertion after renaming the variables.
}
\end{proof}

\begin{prop}
	Let $ K $ be any field and $ n \geq 2 $. 
	For $\cchar(K)\neq 2$ the following holds:
	\begin{enumerate}[leftmargin=*,label=(\arabic*)]
		\item $ \Spec(\cA(B_n)) $ is singular if and only if $n \equiv 3 \mod 4$. 
		
		\item If $n = 4m - 1 $, for some $ m \in \ZZ_{> 0 } $, then $ \Spec(\cA(B_{4m-1})) $ has an isolated singularity 
		{at the origin
		and locally at the singular point, the variety is isomorphic to an $ A_1 $-hypersurface singularity.}
		In particular, its resolution of singularities is given by the blowup of the singular point.
	\end{enumerate}
	
	On the other hand, if $\cchar(K)=2$, then we have:  
	\begin{enumerate}[leftmargin=*,label=(\arabic*')]
		\item 
		If $n = 2m - 1 $, for some $ m \in \ZZ_{> 0 } $, then $ \Spec(\cA(B_{2m-1})) $ has an isolated singularity of type $A_1$ at the origin.  
		
		\item 
		If $ n = 2m $, for some $ m \in \ZZ_{> 0 } $, then $ \Spec(\cA(B_{2m})) $ has an isolated singularity of type $A_1$ at the closed point $ V (u_1 - 1, u_2 -1, u_3, z_1, \ldots, z_{n-1} ) $.

	\end{enumerate}
	In particular, the resolution of singularities is given by the blowup of the singular point in both cases (1') and (2')
	
\end{prop}

\begin{proof} 
	{We use the presentation $\cA(B_n) \cong K[z_1, \ldots, z_{n-1},u_1,u_2,u_3]/\langle g_n, h_n \rangle $ of Lemma~\ref{L:Bn_Pres},
		where 
		$ g_n = (u_1 u_2 - 1 ) u_3 - u_1^2 - P_{n-2}(z_1,\ldots, z_{n-2})  $
		and 
		$ h_n := u_1 u_2 - 1 - P_{n-1}(z_1,\ldots, z_{n-1}) $.
		We consider the $ 2 \times 2 $ minors of the Jacobian matrix.
		The columns corresponding to $ u_1, u_2 , u_3 $ yield that
		the following equations hold for the singular locus
		\[
			u_1 (u_1 u_2 - 1) = u_2 (u_1 u_2 - 1 ) = 2 u_1^2 = 0 
			\ . 
		\] 
		First, assume $ \car(K) \neq 2 $.
		Then, we have $ u_1 = 0 $, which implies $ u_2 = 0 $
		and the $ u_3 $ column of the Jacobian matrix is $ ( -1 , 0 )^T $.
		This implies that the singular locus of $ \Spec(\cA(B_n)) $ is contained in the singular locus of $ P_{n-1}(z_1,\ldots, z_{n-1})  + 1  = 0 $.
		The latter is empty if $ n - 1 \equiv 1 \mod 2 $ 
		(Lemma~\ref{L:Sing(Pn)}(2))
		or if $ n -1 = 2m $ for some $ m \in \ZZ_+ $ and 
		$ m  \equiv 0 \mod 2 $  (Proposition~\ref{Prop:SingPn}).
		Therefore, $ \Spec(\cA(B_n)) $ is {regular} if $ n \not\equiv 3 \mod 4 $.
		\\
		Consider the case $ n \equiv 3 \mod 4 $.  
		Then the singular locus of  $ P_{n-1}(z_1,\ldots, z_{n-1})  + 1   $ is $ V ( z_1, \ldots, z_{n-1}) $ (Proposition~\ref{Prop:SingPn}). 
		In particular, $ P_{n-2}(z_1, \ldots, z_{n-2})  = 0 $, by Lemma~\ref{pnquad}.
		Moreover, $ u_1 = 0 $ implies that $ g_n = 0 $ is equivalent to 
		$ u_3 =  0 $. 
		Hence, $ \Spec(\cA(B_n)) $ has an isolated singularity 
		{at the origin.}
		{Locally at the origin,
		$ u_1 u_2 - 1 $ is invertible and 
		thus the generator	
		$ g_n = (u_1 u_2 - 1 ) u_3 - u_1^2 - P_{n-2}(z_1,\ldots, z_{n-2})  $ can be eliminated.
		This has no effect on $ h_n $ since $ u_3 $ does not appear in it.
		We obtain that, locally at the singular point, $ \Spec(\cA(B_n)) $ is isomorphic to an $ A_1 $-hypersurface singularity},
		where the type of the singularity can be seen by 
		applying the same coordinate transformation as in the proof of Proposition~\ref{Prop:SingPn}. 
	
		Suppose $ \car(K) = 2 $. 
		We make a case distinction for $ u_1 (u_1 u_2 - 1) = 0 $.
		If $ u_1 = 0 $, then the same arguments as for $ \car(K) \neq 2 $ apply,
		which provides an isolated $ A_1 $-singularity at the origin if and only if $ n \equiv 1 \mod 2 $.
		\\
		Thus, let $ u_1 u_2 - 1 = 0 $. 
		Then $ g_n = h_n = 0 $ is equivalent to 
		$ u_1^2 + P_{n-2} (z_1, \ldots, z_{n-2}) = P_{n-1} (z_1, \ldots, z_{n-1}) = 0 $. 
		{The minor of the Jacobian matrix of $ g_n $ and $ h_n $ corresponding to $ ( u_2, z_{n-1} ) $ provides that we have to have $ u_1 u_3 P_{n-2} (z_1, \ldots, z_{n-2} ) = 0 $.
		Since $ u_1 u_2 - 1 = 0 $ and $ P_{n-2} (z_1, \ldots, z_{n-2})  = - u_1^2 $, 
		we obtain $ u_3 = 0 $.} 
		The column of the Jacobian matrix with respect to $ u_2 $ is $ (0,u_1)^T $ 
		and $ u_1 \neq 0 $ provides that all derivatives of $ P_{n-2} (z_1, \ldots, z_{n-2} ) $ have to vanish.
		If $ n \equiv 1 \mod 2 $, then $ V ( P_{n-2} + u_1^2 ) $ is {regular} by Lemma~\ref{L:Sing(Pn)}(2) (where $ u_1^2 \neq 0 $ takes the role of $ \lambda $).
		On the other hand, if $ n \equiv 0 \mod 2 $, then Proposition~\ref{Prop:SingPn} implies that $ V( P_{n-2} + u_1^2 ) $ has a singularity of type $ A_1 $ at $ V (z_1, \ldots, z_{n-2}) $ if and only if $ u_1 = 1 $.
		From this, we obtain assertion (2').}
\end{proof}

\medskip

\subsection{$\boldsymbol{C_n}$ cluster algebras}

We choose the exchange matrix $B$ for type $C_n$, {$n \geq 3$}, as
\[
B=\begin{pmatrix} 0 & 1 & \cdots & 0 & 0 & 0 & 0 \\\
-1 & 0 & \ddots & 0 & 0 & 0 & 0 \\
\vdots & \ddots & \ddots & \ddots & \vdots & \vdots & \vdots  \\
0 & 0 & \ddots & 0 &1 & 0 & 0 \\
0 & 0  & \cdots & -1 & 0 &1 & 0 \\
0 & 0 & \cdots & 0 & -1 & 0 & 2 \\
0 & 0 & \cdots & 0   &0 &-1  & 0 \\

\end{pmatrix} 
\]  
(cf.~\cite[\S 5.5, (5.32)]{FWZ2017}).
The {corresponding Dynkin diagram is} of type $ C_n $
(where $ n $ is the number of vertices); 
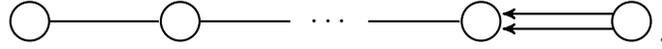
\[
\begin{tikzpicture}[>=stealth',shorten >=1pt,auto,node distance=2cm, thick,main node/.style={circle,draw}]

\node[main node] at (0,0) { \ \ \ };
\node[main node] at (2,0) { \ \ \ };
\node at (4,0) {$ \cdots $};
\node[main node] at (6,0) { \ \ \ };
\node[main node] at (8,0) { \ \ \ };

\draw (0.25,0) -- (1.75,0);
\draw (2.25,0) -- (3.5,0);
\draw (4.5,0) -- (5.75,0);
\draw[->] (7.75,0.1) -- (6.25,0.1);
\draw[->] (7.75,-0.1) -- (6.25,-0.1);

\end{tikzpicture}
\ .
\]

For this cluster algebra, the characteristic of the field makes a significant difference. 
First, we provide a suitable presentation of $ \cA(C_n) $. 

\begin{lemma}
	\label{L:Cn_Pres}
	Let $ K $ be any field and {$ n \geq 3 $}.
	The cluster algebra 
	$ \cA(C_n) $ is isomorphic to 
	\[ 
	K[z_1, \ldots, z_{n+1}]/\langle P_n(z_1, \ldots, z_n)z_{n+1} - P_{n-1}(z_1, \ldots, z_{n-1})^2 -1 \rangle \ ,
	\] 
	where $ P_{n} $ is the continuant polynomial defined in Section~\ref{S:continuants}. 
	In particular, the variety $ \Spec(\cA(C_n)) $ is isomorphic to a hypersurface in $ \mathbb{A}_K^{n+1} $.
\end{lemma}

\begin{proof}
	Since the underlying diagram $ \Gamma(B) $ is acyclic,
	the cluster algebra $ \cA(C_n)$ has a presentation as
	\[ 
	K[x_1, \ldots, x_n, y_1, \ldots, y_n]/\langle 
	h_1, \ldots, h_n \rangle 
	\]
	where we define
	$ h_1 := x_1 y_1-x_2-1  $,
	$ h_n := x_n y_n -x_{n-1}^2-1 $,
	and 
	$ h_k := x_k y_k - x_{k-1} - x_{k+1} $, 
	for $ k \in \{ 2, \ldots, n- 1 \} $.
	Similar as for type $A_n$ (see Lemma \ref{L:presentationAn}) one can for $2 \leq k \leq n$ stepwise express $x_k$ in terms of $x_1, y_1,\ldots, y_{k-1}$:
	$$x_k=P_{k}(x_1, y_1, \ldots, y_{k-1}) \ . $$
	The only difference is in the last generator $ h_n $, which becomes
	$$h_n(x_1, y_1, \ldots, y_n)= P_n(x_1, y_1, \ldots, y_{n-1})y_n - P_{n-1}(x_1, y_1, \ldots, y_{n-2})^2 - 1 \ .$$
	After renaming the variables, we obtain the assertion. 
\end{proof}

\begin{prop}
	Let {$ n \geq 3 $} and $K$ be a field with $\car(K) \neq 2$.
	The {variety} $\Spec(\mathcal{A}(C_n))$ is isomorphic to a {regular} hypersurface in $\A_K^{n+1}$.
\end{prop}

\begin{proof} 
	The proof follows the steps of the proof of Proposition~\ref{Prop:SingPn}.
	We apply the Jacobian criterion to the presentation $\mathcal{A}(C_n) \cong K[z_1, \ldots, z_{n+1}]/\langle h_n(z_1, \ldots, z_{n+1}) \rangle $ of Lemma~\ref{L:Cn_Pres}.
	Since $ \frac{\partial h_n}{\partial z_{n+1}} = P_n (z_1, \ldots, z_n) = 0 $, 
	the equation $ h_n = 0 $ for the singular locus is equivalent to 
	$ P_{n-1}(z_1, \ldots, z_{n-1})^2 = -1 $. 
	If $ -1 \notin K^2 $ is not a square in $ K $, the last equality cannot hold and we have shown that the claim.
	\\
	Assume $ -1 \in K^2 $ and let
	$ \lambda \in K \setminus \{ 0 \} $ be such that $ \lambda^2 = - 1 $ and 
	$ P_{n-1}(z_1, \ldots, z_{n-1}) = \lambda $. 
	Using Lemma~\ref{L:pnder}, we get 
	\[  
	0 = \frac{\partial h_n}{\partial z_n} = P_{n-1} (z_1, \ldots, z_{n-1}) z_{n+1}
	=  \lambda z_{n+1} 
	\]
	and thus $ z_{n+1} = 0 $. 
	\\
	We prove by induction on $ k $ that $ h_n = \frac{\partial h_n}{\partial z_{n+1}} = \ldots  \frac{\partial h_n}{\partial z_{n-2k}} = 0 $ imply
	\begin{enumerate} 
		\item[(a$_k$)] 
		$ z_{n+1} = \cdots = z_{n-2k+1} = 0 $,
		
		\item[(b$_k$)] 
		$ P_{n-2k-1}(z_1, \ldots, z_{n-2k-1}) = (-1)^{k} \lambda $,
		and 
		
		\item[(c$_k$)] 
		$ P_{n-2k}(z_{1}, \ldots, z_{n-2k}) =  0 $.
	\end{enumerate}
	As we have seen above, the statements are true for $ k = 0 $. 
	Let us deduce  (a$_{k+1}$), (b$_{k+1}$), (c$_{k+1}$) from (a$_{k}$), (b$_{k}$), (c$_{k}$).
	First, (a$_k$) and Lemma~\ref{pnquad} imply
	$ P_{2k}(z_{n-2k}, \ldots, z_{n-1}) = \pm 1  $.
	Using $ z_{n+1} = 0 $, (b$_0$), and Lemma~\ref{L:pnder},  
	we have
	\[
	0 = \frac{\partial h_n}{\partial z_{n-2k-1}}
	= - 2 P_{n-1} (z_1, \ldots, z_{n-1}) \frac{\partial P_{n-1} (z_1, \ldots, z_{n-1}) }{\partial z_{n-2k-1}} =
	\]
	\[
	= (-2) \lambda P_{n-2k-2} (z_1, \ldots, z_{n-2k-2})
	P_{2k}(z_{n-2k}, \ldots, z_{n-1}) = 
	( \mp 2) \lambda P_{n-2k-2} (z_1, \ldots, z_{n-2k-2}) 
	\ ,
	\]
	which provides (c$_{k+1} $) as $ \car(K) \neq 2 $. 
	If we apply the recursion of Lemma~\ref{L:pnrecur} for $ P_{n-2k}(z_{1}, \ldots, z_{n-2k}) $ and use (b$_k$) and (c$_{k+1} $),
	we obtain $ z_{n-2k} = 0 $. 
	On the other hand, Lemma~\ref{L:pnrecur} applied for $ P_{n-2k-1} (z_1, \ldots, z_{n-2k-1}) $, 
	(b$_{k} $),
	and (c$_{k+1} $) provide (b$_{k+1} $).  
	It remains to prove $ z_{n-2k-1} = 0 $. 
	For this, we consider the derivative of $ h_n $ by $ z_{n-2k-2} $
	(using Lemma~\ref{L:pnder}) and apply (b$_0$), (b$_{k+1}$), 
	as well as $ P_{2k+1} (z_{n-2k-1}, 0, \ldots, 0 ) = \pm z_{n-2k-1} $ (by Lemma~\ref{pnquad}). 
	
	Next, we distinguish two cases:
	If $ n = 2m $ for some $ m \in \ZZ_+$,
	then we have not used the derivative by $ z_1 $ in the induction. 
	Using $ z_{n+1} = 0 $ and (b$_0$), we obtain the contradiction 
	\[
	0 = 
	\frac{\partial h_n}{\partial z_1} 
	=
	-2 \lambda P_{2m-2} (z_2, z_3,\ldots, z_{n-1}) = \mp \lambda \neq 0 
	\ ,   
	\]
	where the last equality holds by applying (a$_{m-1}$) and Lemma~\ref{pnquad}.
	\\
	Assume $ n = 2m +1 $, for some $ m \in \ZZ_+ $. 
	Statement (a$_{m}$) implies $ z_2 = 0 $
	and (b$_{m-1}$) states $ P_2 (z_1,z_2) = \pm \lambda $. 
	This leads to the equality $ \lambda = \mp 1 $,
	which contradicts the property $ \lambda^2 = - 1 $ 
	(see the definition of $ \lambda $ above).  
\end{proof}

Let us discuss the case $ \cchar(K) = 2 $, 
which turns out to be more complicated.

\begin{prop}
	\label{Prop:Cn_p2}
	Let $ K $ be a field of characteristic two. 
	The singular locus of 
	$\Spec(\cA(C_n))  $
	is isomorphic to $ \Spec(\cA(A_{n-2})) $ and is of dimension $ n - 2 $. 
	Moreover, we have: 	
	\begin{enumerate}[leftmargin=*,label=(\arabic*)]
		\item 
		If $ n \equiv 0 \mod 2 $, then $ \Sing(\cA(C_n)) $ is {regular} and 
		$ \Spec(\mathcal{A}(C_n)) $ is locally at the singular locus isomorphic to a cylinder over a $ A_1 $-hypersurface singularity in $ \A_K^3 $. 
		In particular, the blowup with center 
		$ \Sing(\cA(C_n)) $ resolves the singularities of 
		$ \Spec(\mathcal{A}(C_n)) $.

		\item 
		If $ n \equiv 1 \mod 2 $ {and $ n > 3 $},
		then $ \Sing(\cA(C_n)) $ has an isolated singularity of type $ A_1 $ at the origin. 
		Locally at the origin, 
		$ \Spec(\mathcal{A}(C_n)) $ is isomorphic to a hypersurface singularity of the form
		{(where $ m $ is defined by $ n = 2m+1 $)}
		\[  
		\Spec ( k[x_1, \ldots, x_{2m}, y, z] / \langle \, yz + \big( \sum_{i=1}^{m} x_{2i-1} x_{2i} \big)^2  \ \rangle 
		\ ,
		\ \ \ \mbox{where } n = 2m +1, 
		\]
		while locally at a singular point different from the origin,
		$ \Spec(\mathcal{A}(C_n)) $ is again isomorphic to a cylinder over the $ A_1 $-hypersurface singularity given by $ V(x^2 + yz)  \subset \A_K^3 $. 
		\\
		The singularity $ \Spec(\mathcal{A}(C_n)) $ is resolved by three blowups;
		the first center is the origin,
		the second is the strict transform of the original singular locus,
		and 
		the third center is the strict transform of an exceptional component created after the first blowup.
		
		\item 
		If $ n = 3 $,
			then $ \Sing(\cA(C_n)) $ is isomorphic to two {regular} lines intersecting transversally at the origin. 
			All other statements of (2) remain true for $ m = 1 $.
	\end{enumerate}
\end{prop}

\begin{Bem}
	Observe that case (2) is not among the simple singularities.
\end{Bem}

\begin{proof}[Proof of Proposition~\ref{Prop:Cn_p2}]
	As we have seen above, $ \cA(C_n) $ is isomorphic to the hypersurface determined by 
	$ h_n (z_1, \ldots, z_{n+1}) = P_n(z_1, \ldots, z_n)z_{n+1} - P_{n-1}(z_1, \ldots, z_{n-1})^2 -1 $.
	Since $ \cchar(K) = 2 $,
	we can rewrite this as
	\[
	h_n (z_1, \ldots, z_{n+1}) = P_n(z_1, \ldots, z_n)z_{n+1} + \big( 
	P_{n-1}(z_1, \ldots, z_{n-1}) + 1 \big)^2
	\ .
	\]
	The partial derivatives are 
	\[
	\begin{array}{rcl}
	\displaystyle
	\frac{\partial h_n}{\partial z_{n+1}}
	& = & 
	P_n(z_1, \ldots, z_n)
	\ , 
	\\[12pt] 
	\displaystyle
	\frac{\partial h_n}{\partial z_{k}}
	& = & 
	\displaystyle 
	z_{n+1} \frac{\partial P_n(z_1, \ldots, z_n) }{\partial z_k} 
	\ ,
	\ \ \ \mbox{ for }  k \in \{ 1, \ldots, n \}
	\ .
	\end{array}
	\]
	This implies that we can replace the condition $ h_n(z_1, \ldots, z_{n+1}) = 0 $ by 
	\[  
	f_{n-2}(z_1, \ldots, z_{n-1}) = P_{n-1}(z_1, \ldots, z_{n-1}) + 1   = 0 
	\] 
	when determining the singular locus,
	where $ f_{n-2}(z_1, \ldots, z_{n-1}) $ is the polynomial
	providing a hypersurface presentation for the {variety} $ \Spec(\cA(A_{n-2})) $
	(see Lemma~\ref{L:presentationAn} and use $ \cchar(K) = 2 $).
	\\
	Since the second factor of 
	$ \frac{\partial h_n}{\partial z_{n}} = z_{n+1} P_{n-1} (z_1, \ldots, z_n) $ 
	cannot vanish if $ f_{n-2}  = 0 $
	and since $ P_n (z_1, \ldots, z_n) = z_n P_{n-1} (z_1, \ldots, z_{n-1}) + P_{n-2} (z_1, \ldots, z_{n-2}) $,
	we obtain 
	\[
	\Sing (\cA(C_n)) 
	= V (u_n, z_{n+1}, f_{n-2}(z_1, \ldots, z_{n-1})) =:D
	\ , 
	\]
	where we introduce $ u_n := z_n + P_{n-2} (z_1, \ldots, z_{n-2}) $. 
	Observe that
	\[ 		
	h_n = f_{n-1}^2 + z_{n+1} \Big( u_n (1 + f_{n-1}) + f_{n-1} P_{n-2} \Big) \ .
	\]

	Corollary~\ref{Cor:An}(1') implies that $ \Sing (\cA(C_n)) $ is singular if and only if $ n \equiv 1 \mod 2 $.
	\\
	First assume $ n \equiv 0 \mod 2 $. 
	Then $ V (f_{n-1}) $ is {regular} and we may take $ u_{n-1} := f_{n-1} $ as a local variable locally at the singular locus. 
	Furthermore, $ 1 + f_{n-1} $ is a unit and
	we may introduce the local variable $ w_n :=  u_n (1 + f_{n-1}) + f_{n-1} P_{n-2} $.
	Hence, $ h_n = u_{n-1}^2 + z_{n+1} w_{n} $ and claim (1) follows. 
	\\
	Suppose that $ n \equiv 1 \mod 2 $.
	By Corollary~\ref{Cor:An}(2'), 
	$ \Sing (\cA(C_n)) $ has an isolated singularity of type $ A_1 $ at the origin $ V (z_1, \ldots, z_{n-1}, u_n, z_{n+1} ) $.
	The statement about the the type of singularity away from origin follows with the same argument as in the case $ n \equiv 0 \mod 2 $.
	\\
	Let us study the situation at the origin.
	Write $ n - 2 = 2m - 1 $, for $ m \in \ZZ_+ $.
	As we have seen in the proof of Proposition~\ref{Prop:SingPn},
	there is a coordinate transformation $ (z_1, \ldots, z_{2m}) \mapsto (t_1, \ldots, t_{2m}) $ such that  
	$ f_{n-2} (t_1, \ldots, t_{2m}) = \sum_{i=1}^{2m} t_{2i-1} t_{2i} $, 
	locally around the origin.
	Furthermore, we can again introduce $ w_n $ above such that, 
	locally at the origin, we obtain 
	$ h_n = ( \sum_{i=1}^{m} t_{2i-1} t_{2i}  )^2 + w_n z_{n+1} $, as desired. 
	
	Let us discuss the desingularization of the {variety} $ \Spec(\cA(C_n)) $. 
	First, we blow up with center the origin.
	In order to simplify the presentation of the charts, we abuse notation and write $ z_n := u_n $.
	
	\change{We fix the notation when considering explicit charts of a blowup: 
		(We only discuss this for the blowup of the origin, but it can be adapted for any blowup with a smooth center.) 
		Recall that the blow-up in the origin of $ \A^{n+1}_K = \Spec(R) $,
		where $ R := K[z_1, \ldots, z_{n+1}] $,
		 is given by $ \operatorname{Bl}_0(\A^{n+1}_K) := \operatorname{Proj}(R[Z_1, \ldots, Z_{n+1}]/ \langle z_i Z_j - z_j Z_i \mid i, j \in \{ 1, \ldots, n +1 \}  \rangle ) $,
	 where $ (Z_1, \ldots, Z_{n+1} ) $ are projective variables.
	 In particular, $ \operatorname{Bl}_0(\A^{n+1}_K) $ is covered by the open subsets $ D_i $  given by $ Z_i \neq 0 $, for $ i \in \{ 1, \ldots, n + 1 \} $.
 	We also say that $ D_i $ is the {\em $ Z_i $-chart}.
 	\\
 	Fix $ i \in \{ 1, \ldots, n +1 \} $. 
 	Since $ z_i Z_j - z_j Z_i = 0 $, for $ j \neq i $, we obtain that
 	$ z_j = z_i \frac{Z_j}{Z_i} $ in the $ Z_i $-chart.
 	This provides that the $ Z_i $-chart is isomorphic to $ \Spec (K[z_1', \ldots, z_{n+1}'] ) $, where we set $ z_i' := z_i $ and $ z_j' := \frac{Z_j}{Z_i} $ for $ j \neq i $.
	 In order to keep the notation light, we abuse it by using the same letter for the variables after the blowup as before, i.e., the transformation of the variables will be written as $ z_j = z_i z_j $ for $ j \neq i $.}
	
	\noindent
	\underline{$ Z_n $-chart.}
	We have $ z_i = z_n z_i $ for every $ i \neq n $.
	The strict transform of $ h_n $ is 
	\[
	h_n' = f_{n-1}'^2 z_n^2  + z_{n+1} ( 1 + z_n^2 f_{n-1}' + z_n^2 f_{n-1}' P_{n-2}' ) \ ,
	\] 
	where we denote by $ f_{n-1}' $ (resp.~$ P_{n-2}' $) the strict transform of $ f_{n-2} $ 
	(resp.~$ P_{n-2} $). 
	The strict transform $ D' $ of $ D $ is empty in this chart.
	Hence, the only singularities which may appear have to be contained in the exceptional divisor $ V (z_n) $. 
	On the other hand, we have 
	$ \frac{\partial h_n'}{\partial z_{n+1}'} = 1 + z_n^2 f_{n-1}' + z_n^2 f_{n-1}' P_{n-2}'  $,
	which 
	implies that the strict transform of $ \Spec (\cA (C_n)) $ is {regular} in this chart.
	
	The analogous argument applies for the $ Z_{n+1} $-chart.
	
	\noindent
	\underline{$ Z_1 $-chart.}
	(All other charts remaining are analogous.)
	We get $ z_i = z_1 z_i $ for every $ i \neq 1 $
	and 
	\[
	h_n' = f_{n-1}'^2 z_1^2  + z_{n+1} \Big( z_n ( 1 + z_1^2 f_{n-1}')  + z_1^2 f_{n-1}' P_{n-2}' \Big) \ .
	\] 
	By Corollary~\ref{Cor:An}(2') $ V(f_{n-1}') $ is {regular} and 
	thus the same is true for $ D' $.
	Lemma~\ref{pnquad} implies that 
	$ f_{n-1}' = z_2 + z_4 + \cdots + z_{n-1} + H $ 
	for some $ H \in \langle z_1, z_2, \ldots, z_{n-1} \rangle^2 $.
	In particular, $ V(f_{n-1}') $ is transversal to $ V(z_1 z_n z_{n+1} ) $
	and we may introduce the variable 
	$ u_2 := f_{n-1}' $ locally at $ D' $.    
	\\
	Since any newly created singularities have to be contained in the exceptional divisor $ V(z_1) $, 
	the element $ 1 + z_1^2 f_{n-1}' $ 
	is invertible locally at the singular locus of $ V ( h_n' ) $.
	We introduce the local variable $ w_n := z_n ( 1 + z_1^2 f_{n-1}')  + z_1^2 f_{n-1}' P_{n-2}' $
	and we get 
	\[
	h_n' = f_{n-1}'^2 z_1^2 + w_n z_{n+1}.
	\]  
	Obviously, we have 
	$ \Sing (V(h_n')) = D' \cup V (z_1, w_n, z_{n+1} ) $. 
	Both components are {regular} and $ V(h_n') $ is an $ A_1 $-singularity at every point expect their intersection.
	Let us define $ E := V (z_1, w_n, z_{n+1} ) $. 
	\\
	Next, we blow up with center $ D' $.
	Observe that this is a well-defined global center, which is seen in any $ Z_i $-chart with  $i \in \{ 1 ,\ldots, n-1 \} $
	as it is the strict transform of $ \Sing(\Spec(\cA(C_n))) $. 
	There are no new singularities contained in the exceptional divisor of the second blowup
	and hence the singular locus of $ V(h_n'') $ has to be the strict transform $ E '' $ of $ E $
	(where $ h_n'' $ is the strict transform of $ h_n' $ after the second blowup).
	Locally at $ E'' $, the hypersurface is given by an equation of the form $ x^2 - yz = 0 $.
	Therefore, after blowing up $ E''$ all singularities are resolved.
	Again observe that $ E'' $ is a well-defined global center at this step of the resolution process since it is the singular locus of the strict transform after the second blowup. \\
	Case (3) is seen by explicit computation.
\end{proof}

\medskip

\subsection{$\boldsymbol{F_4}$ and $\boldsymbol{G_2}$ cluster algebras}

Let us discuss the remaining cases of $ \cA(F_4) $ and $ \cA(G_2 ) $.

For the cluster algebra $ \cA(F_4) $, we pick the exchange matrix $ B = \begin{pmatrix} 0 & 1 & 0 &0  \\ -1 & 0 &1 & 0 \\ 0 & -2 & 0 & 1\\ 0 & 0 & -1 & 0 \end{pmatrix}$ 
(cf.~\cite[Exercise~4.4.12]{FWZ2017}),
whose {corresponding Dynkin diagram is} of type $F_4$; 
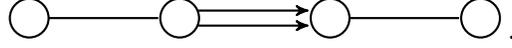
\[
\begin{tikzpicture}[>=stealth',shorten >=1pt,auto,node distance=2cm, thick,main node/.style={circle,draw}]

\node[main node] at (0,0) { \ \ \ };
\node[main node] at (2,0) { \ \ \ };
\node[main node] at (4,0) { \ \ \ };
\node[main node] at (6,0) { \ \ \ };

\draw (0.25,0) -- (1.75,0);
\draw (4.25,0) -- (5.75,0);
\draw[->] (2.25,0.1) -- (3.75,0.1);
\draw[->] (2.25,-0.1) -- (3.75,-0.1);

\end{tikzpicture}
\ .
\]

\begin{lemma}
	\label{Lem:F4}
	For any field $ K $, 
	the {variety} $ \Spec (\cA (F_4)) $ is isomorphic to a {regular} hypersurface in $ \A_K^5 $.
\end{lemma} 

\begin{proof} 
	The underlying graph $ \Gamma(B) $ is acyclic and 
	thus $ \cA(F_4 ) $ is isomorphic to its lower bound cluster algebra.
	Similar as above we obtain the presentation
	\[ 
	\cA(F_4) \cong 
	K[x_1, \ldots, x_4, y_1, \ldots, y_4]/ \langle x_1y_1 -x_2 -1, x_2y_2 -x_1 -x^2_3, x_3y_3 -x_2 -x_4, x_4y_4-x_3-1 \rangle \ . 
	\]
	By eliminating $ x_1, x_2, x_3 $, this can be simplified to
	\[
	\cA(F_4) \cong K[x,y,z,w,t]/\langle  xyzwt-x^2yt^2-xyz-yzw+2xyt-xwt+x-y+w-1 \rangle  \ .
	\]
	Hence, $\Spec(\mathcal{A}(F_4))$ is isomorphic to a hypersurface in $\A^5_K$.
	Moreover, the Jacobian criterion shows that the latter is {regular}.
\end{proof}

\medskip

Next, let us come to $ \cA(G_2 ) $. 
A possible exchange matrix is $ B = \begin{pmatrix} 0 & 1 \\ -3 & 0 \end{pmatrix} $
(cf.~\cite[\S~5.7]{FWZ2017}) 
and the {corresponding Dynkin diagram is} of type $G_2$; 
\[
\begin{tikzpicture}[->,>=stealth',shorten >=1pt,auto,node distance=2cm, thick,main node/.style={circle,draw}]

\node[main node] at (0,0) { \ \ \ };
\node[main node] at (2,0) { \ \ \ };

\draw (0.25,0) -- (1.75,0);
\draw (0.25,0.15) -- (1.75,0.15);
\draw (0.25,-0.15) -- (1.75,-0.15);

\end{tikzpicture}
\ .
\]

\begin{lemma}
	\label{Lem:G2}
	Let $ K $ be any field. 
	The {variety} $ \Spec (\cA (G_2)) $ is isomorphic to a hypersurface in $ \A_K^3 $.
	\begin{enumerate}[leftmargin=*,label=(\arabic*)]
		\item 
		If $ \car(K) \neq 3 $, then $ \Spec ( \cA(G_2)) $ is {regular}.

		\item 
		If $ \car(K) = 3 $, then $ \Spec ( \cA(G_2)) $ has an isolated singularity of type $ A_2 $ at a closed point.
		In particular, the singularities of the {variety} are resolved by two point blowups.  
	\end{enumerate}
\end{lemma}

\begin{proof}
	Since $ \Gamma(B) $ is acyclic, 
	the cluster algebra $\mathcal{A}(G_2)$ is isomorphic to its lower bound cluster algebra by
	Theorem~\ref{Thm:clus=low}.  
	We get the presentation 
	\[ 
	\cA(G_2) \cong K[x_1,x_2,y_1,y_2]/\langle x_1y_1-1-x_2^3,x_2y_2-1-x_1\rangle \ .
	\] 
	Since $x_1$ can be expressed in term of the other variables, we find
	\[ 
	\cA(G_2)\cong K[x,y,z]/\langle z^3-xyz+y+1\rangle \ . 
	\]
	By the Jacobian criterion, this algebra is {regular} for $\mathrm{char}(K) \neq 3$, and for characteristic $3$ the Jacobian criterion yields 
	$\Sing(\cA(G_2)) \cong 
	V(x+1,y,z+1) $, 
	a closed point.
	\\
	Assume $\mathrm{char}(K)=3$ and set $x':=x+1$, and $z':=z+1$. 
	Then, we obtain $\mathcal{A}(G_2)\cong K[x',y,z']/\langle z'^3-x'yz'+yz'+x'y\rangle$, where the singular point is the origin in the new coordinates. Applying the local coordinate change $x'=\frac{\tilde x -z'}{1-z'}$ yields 
	\[ 
	\cA(G_2)_{\langle \tilde x, y, z' \rangle} \cong K[\tilde x, y, z']_{\langle \tilde x, y, z' \rangle} /\langle z'^3+\tilde x y \rangle
	\ ,
	\] 
	which is a singularity of type $ A_2 $,	
	see \cite[Definition~1.2]{GreuelKroening}.
\end{proof}

\bigskip

\section{Star cluster algebras}
\label{sec:star}

The final section is devoted to the question how singularities of cluster algebras may look 
for more general quivers. 
We examine the cluster algebra of a star shaped quiver.
Consider the \emph{star quiver }$\St_n$ with $n$ vertices and {the following} orientation: 
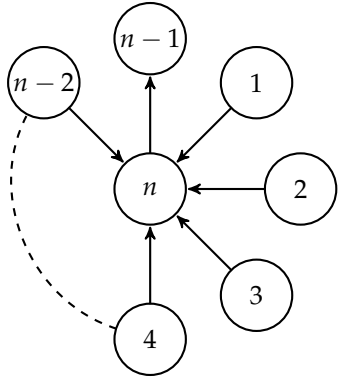
\begin{center} 
	\begin{tikzpicture}[->,>=stealth',shorten >=1pt,auto,node distance=2cm, thick,main node/.style={circle,draw,white}]

\node[main node] (n) {1 1};
\node[main node] (1) [above right of=n] {1 1};
\node[main node] (2) [right of=n] {1 1};
\node[main node] (3) [below right of=n] {1 1};
\node[main node] (4) [below of=n] {1 1};
\node[main node] (n-2) [above left of=n] {1 1};
\node[main node] (n-1) [above  of=n] {1 1};

\node at (n) {\small $n$};
\node at (1) [above right of=n] {\small $ 1 $};
\node at (2) [right of=n] {\small $2$};
\node at (3) [below right of=n] {\small $3$};
\node at (4) [below of=n] {\small $4$};
\node at (n-2) [above left of=n] {{\small $n-2$}};
\node at (n-1) [above  of=n] {{\small $n-1$}};

\draw (n) circle (13.5pt);
\draw (1) circle (13.5pt);
\draw (2) circle (13.5pt);
\draw (3) circle (13.5pt);
\draw (4) circle (13.5pt);
\draw (n-2) circle (13.5pt);
\draw (n-1) circle (13.5pt);

\path

(3) edge (n)
(1) edge (n)
(2) edge (n)
(4) edge (n)
(n-2) edge (n)
(n) edge (n-1);

\path[dashed]
(4) edge[bend left=50,-] (n-2);

\end{tikzpicture}
\end{center} 
By Lemma~\ref{Lem:tree_equiv}, 
all orientations on this tree are equivalent. 
Since the quiver is acyclic (and all vertices are mutable),
the corresponding cluster algebra,
denoted by $ \cA(\St_n) $,
is isomorphic to its lower bound cluster algebra by Theorem~\ref{Thm:clus=low}. 

\begin{Bem}
	\label{Rk:Sn_known} 
Observe that $ \St_2 = A_2 $, $ \St_3 = A_3 $, and $ \St_4 =D _4 $. 
By Corollary~\ref{Cor:An},
$ \Spec (\cA(\St_2)) $ is {regular},
while $ \Spec(\cA(\St_3)) $ is an $ A_1 $-hypersurface singularity
which is resolved by blowing up the singular locus.
Moreover, by Lemma~\ref{Lem:SingDn} and Proposition~\ref{Prop:TypeRSDn}, 
the singular locus of $ \Spec (\cA(\St_4))  $ has six irreducible components of dimension one and we obtain a desingularization of  $ \Spec (\cA(\St_4))  $ by first blowing up the intersection of these irreducible components, followed by the blowup of their strict transforms.

In particular, if suffices to restrict to the case $ n \geq 5 $ in the following, if needed. 
\end{Bem}  

\begin{lemma}
	\label{L:Sn} 
	The cluster algebra $\mc{A}(\St_n)$ has a presentation of the form
	\[ 
		K[z_1,\ldots,z_{2n-2}]
		/ \langle h_1, h_2, \ldots, h_{n-2} \rangle 
		\ ,
	\]
	where
	\[ 
		h_1 := z_1 z_2 (1  - z_{2n-3} z_{2n-2}  + \prod_{\ell=1}^{n-1} z_{2\ell-1} )
		+ z_{2n-3} z_{2n-2}  \ ,
	\]
	\[
		h_k := z_1 z_2 - z_{2k-1} z_{2k} 
		\  , \ \ \ \mbox{ for } k \in \{ 2, \ldots, n-2 \} \ .
	\]
	In particular, the $ n $-dimensional variety 
	$\Spec(\mc{A}(\St_n)) $ can be embedded into $ \A^{2n-2}_K $.
\end{lemma}

\begin{proof}
	Since $ \cA(\St_n) $ is isomorphic to its lower bound cluster algebra,
	we have
	\[
		\cA(\St_n) \cong K[x_1, \ldots, x_n, y_1, \ldots, y_n]/  I \ ,
	\]
	where $ I $ is the ideal generated by the exchange relations
	\[
	\begin{array}{llll} 
	x_1 y_1 - x_n  - 1 
	\ , 
	&
	\ldots
	\ , 
	&
	x_{n-1} y_{n-1} - x_n - 1
	\ , 
	&
	x_n y_n - x_1 \cdots x_{n-2} - x_{n-1}
	\ .
	\end{array} 
	\] 
	The first generator allows us to substitute $ x_{n} = x_1 y_1 - 1 $
	and we get
	\[
	\begin{array}{lllll} 
	x_1 y_1 - x_{2} y_{2} 
	\ , 
	&
	\ldots
	\ , 
	&
	x_1 y_1 - x_{n-1} y_{n-1}
	\ , 
	&
	y_n (x_1 y_1 - 1) - x_1 \cdots x_{n-2} - x_{n-1}
	\ .
	\end{array}	
	\] 
	We can eliminate another generator via $ x_{n-1} = y_n (x_1 y_1 - 1) - x_1 \cdots x_{n-2} $ , which provides
	\[
	\begin{array}{rcl} 
	h_1 := x_1 y_1 - x_{n-1} y_{n-1} 
	& 
	= 
	& 
	x_1 y_1 - y_{n-1} (y_n (x_1 y_1 - 1) - x_1 \cdots x_{n-2})
	\\
	& 
	= 
	& 
	x_1 y_1 (1 - y_{n-1} y_n) 
	+ y_{n-1} (y_n +  x_1 \cdots x_{n-2} ) 
	\ , 
	\end{array}	
	\] 
	while $ h_k :=  x_1 y_1 - x_k y_k $, for $ k \in \{ 2, \ldots, n- 2 \} $, are unchanged.
	Introducing
	$ z_{2n-3} := y_{n-1} $,
	$ z_{2n-2} := y_n + x_1 \cdots x_{n-2} $,
	and  
	$ z_{2k-1} := x_k $, $ z_{2k} :=  y_k $, for $ k \in \{ 1, \ldots, n-2\} $,
	provides the assertion.
\end{proof}

\begin{Thm}
	\label{Thm:Sn}
		Let $ n \geq 4 $.
		Let $ \cA(\St_n) $ be the cluster algebra arising form the star shaped quiver $ \St_n $ over a field $ K $. 
		Using the notation of Lemma~\ref{L:Sn}, 
		we have 
		\[ 
			\Sing ( \Spec(\mc{A}(\St_n)))  
			\cong
			\bigcup_{k = 1 }^{n-1} 
			\bigcup_{\ell = k+1}^{n-1} D_{k,\ell} 			
			\subseteq \A_K^{2n-2} 
			\ ,
		\]
		\[
			\mbox{for } 
			D_{k,\ell} := V ( z_{2k-1}, \, z_{2k}, \, z_{2\ell -1}, \, z_{2\ell}, \,
			z_{2m-1} \cdot z_{2m} \mid m \in \{ 1, \ldots, n-1 \} \setminus \{ k , \ell \} )
			\ . 
		\]
		
		In particular, the singular locus consists of $ \binom{n-1}{2} 2^{n-3} = (n-1)(n-2)2^{n-4} $ irreducible components, 
		where each of them is {regular} and of dimension $ n - 3 $.
		Furthermore, locally at a generic point of such a component, $ \Spec (\cA(\St_n) ) $ is isomorphic to an $ A_1 $-hypersurface singularity.
		On the other hand, locally at the closed point determined by the intersection of all irreducible components, 
			$ \change{\Spec (\cA(\St_n) )} $ is isomorphic to a toric variety,
			defined by the binomial ideal
			\[
			\langle x_1 x_2 - x_{2k-1} x_{2k} \mid k \in \{ 2, \ldots, n - 1 \} \rangle \subset K [x_1, \ldots, x_{2n-2}]_{ \langle x_1, \ldots, x_{2n-2} \rangle }.
		\]
		We have that the intersection of all irreducible components is the origin of $ \A_K^{2n-2} $ and after blowing up the latter, 
		we obtain in each chart a singularity which is of the same kind as the one of $ \Spec(\cA(\St_{n-1})) \subset \A_K^{2n-4} $.
		In other words, the singularities of $ \Spec(\cA(\St_n)) $ are resolved by first separating the irreducible components of its singular locus and then blowing up their strict transforms.  
\end{Thm}

As a preparation for the proof, we show the following lemma,
which we will use to make an induction.

\begin{lemma}	
	\label{L:auxSing} 
	Let $ h := x_1 x_2 (1 + \rho ) + y_1 y_2  \in K[x_1, x_2, y_1, y_2, u_1, \ldots, u_a] $, for $ \rho \in \langle y_1 \rangle $.
	We have:
	\begin{enumerate}[leftmargin=*,label=(\arabic*)]
		\item 
		$ \Sing ( V(h)) \cap V( x_1 x_2 ) = V (x_1, x_2, y_1, y_2 ) $. 
		
		\item 
		If $ \rho = y_1 y_2 \rho_1 +  x_1 y_1 \rho_2  $, for $ \rho_1 , \rho_2 \in K[u_1, \ldots, u_a] $,
		then $ \Sing ( V(h))  = V (x_1, x_2, y_1, y_2 ) $. 
		
	\end{enumerate}
	Moreover, locally at
	$  V (x_1, x_2, y_1, y_2 ) $, 
	the variety $ V(h) $ is a cylinder over an $ A_1 $-hypersurface singularity.
	In particular, blowing up $ V (x_1, x_2, y_1, y_2 )  $ resolves the singularities 
	{of $ V(h) $}.  
\end{lemma}

\begin{proof} 
	Notice that the inclusion $ V (x_1, x_2, y_1, y_2 ) \subseteq \Sing( V(h)) $ is obvious. 
	It remains to show the equality.
	We have
	\[
	\begin{array}{ll} 
	\dfrac{\partial h}{\partial x_1} 
	= x_2 (1 + \rho ) + x_1 x_2 \dfrac{\partial \rho}{\partial x_1} 
	\ ,
	& 
	\dfrac{\partial h}{\partial x_2}
	= x_1 (1 + \rho ) + x_1 x_2 \dfrac{\partial \rho}{\partial x_2}
	\ , 
	\\[10pt] 
	\dfrac{\partial h}{\partial y_1} 
	= x_1 x_2 \dfrac{\partial \rho}{\partial y_1} + y_2
	
	\ ,
	& 
	\dfrac{\partial h}{\partial y_2}
	= x_1 x_2 \dfrac{\partial \rho}{\partial y_2} + y_1
	\ . 
	\end{array} 
	\] 
	If $ x_1 x_2 = 0 $, then the vanishing of the derivatives with respect to $ y_1 $ and $ y_2 $ implies $ y_1 = y_2 = 0 $.
	Since $ \rho \in \langle y_1 \rangle $,we obtain $ x_1 = x_2 = 0$ from the other two derivatives. 
	Thus, all components of the singular locus 
	with $ x_1 x_2 = 0  $ are contained in $ V(x_1, x_2, y_1, y_2) $. 
	\\
	Consider the case 
	$ \rho = y_1 y_2 \rho_1 +  x_1 y_1 \rho_2  $ and $ x_1 x_2 \neq 0 $.
	Then, we have
	$ \frac{\partial h}{\partial x_2}
	= x_1 (1 + \rho ) $ and
	we must have $ 1 + \rho = 0 $.
	This provides $ y_1 y_2 = 0 $ since $ h = 0 $.
	But $ y_1 = 0 $ is impossible as this would contradict $ 1 + \rho = 0 $.
	We get $ y_2 = 0 $ and $ 0 = 1 + \rho = 1 + x_1 y_1 \rho_2 $.
	The vanishing of $ \frac{\partial h}{\partial y_1} $ leads to the condition
	$
	0 = \frac{\partial \rho}{\partial y_1} = x_1 \rho_2  
	$, 
	which contradicts  $ 1 + x_1 y_1 \rho_2   = 0 $.
	
	Finally, 
	locally at a point of $ V (x_1, x_2, y_1, y_2) $, the element $ 1 + \rho $ is a unit since $ \rho \in \langle y_1 \rangle $ and hence $ V(h) $ is locally isomorphic to the hypersurface singularity $ V (x_1 x_2 + y_1 y_2) \subset \A_K^{4+a} $.
	This implies the remaining statements.
\end{proof} 

\begin{proof}[Proof of Theorem~\ref{Thm:Sn}]
	By Lemma~\ref{L:Sn},
	we have 
	$\mc{A}(\St_n)  
	\cong 
	K[z_1,\ldots,z_{2n-2}]
	/ I $,
	where $ I := \langle h_1, h_2, \ldots, h_{n-2} \rangle $ and
	\[ 
	\begin{array}{l}
	\displaystyle  
		h_1 = z_1 z_2 (1  - z_{2n-3} z_{2n-2}  + \prod_{\ell=1}^{n-1} z_{2\ell-1} )
		+ z_{2n-3} z_{2n-2}  \ ,
		\\[15pt]
		h_k = z_1 z_2 - z_{2k-1} z_{2k} 
		\  , \ \ \ \mbox{ for } k \in \{ 2, \ldots, n-2 \} \ .
	\end{array}
	\]	
	Observe that each generator is of the form as $ h $ in Lemma~\ref{L:auxSing}.
	Furthermore, for every $ \ell \in \{ 2, \ldots, n-2 \} $ fixed,
	we may interchange the role of $ z_1 z_2 $ and $ z_{2\ell-1} z_{2\ell} $ using the relation $ h_\ell = 0 $.
	Hence, Lemma~\ref{L:auxSing} implies that	
	\begin{equation}  
	\label{eq:D_Sn}
		D := \bigcup_{k = 1 }^{n-1} 
		\bigcup_{\substack{\ell = 1\\[2pt] \ell \neq k }}^{n-1} D_{k,\ell} 				
		\subseteq 
		\Sing ( \Spec(\mc{A}(\St_n)))  
		\ .
	\end{equation} 
	It remains to prove that this is an equality.
	Suppose there exists $ C \subset \Sing( \Spec(\mc{A}(\St_n)) ) $ with $ C \nsubseteq D $.
	We deduce a contradiction via an induction on the number of generators $ h_1, \ldots, h_{n-2} $.	 
	If $ n-2 = 2 $, 
	then $ n = 4 $, i.e., $ S_4 = D_4 $ (by Remark~\ref{Rk:Sn_known}).
	By Lemma~\ref{Lem:SingDn}, the singular locus of $ \Spec(\cA(D_4)) $ consists of the six lines determined by $ D_{k,\ell} $.
	\\
	Suppose $ n -2 > 2 $. 
	The {\em induction hypothesis} implies that \eqref{eq:D_Sn} is an equality for any $ K[z_1, \ldots, z_{2m-2}, u_1, \ldots, u_a]/\langle g_1, \ldots, g_{m-2} \rangle $ 
	with $ m-2 < n- 2 $ and 
	\[ 
		\begin{array}{l}
		g_1 = z_1 z_2 (1 + \rho ) + z_{2m-3} z_{2m-2} \ , 
		\\[5pt]
		g_k = z_1 z_2 - z_{2k-1} z_{2k} \ , \ \ \mbox{ for } k \in \{ 2, \ldots, m- 2 \}
		\ ,
		\end{array} 
	\] 
	where $ \rho \in \langle z_{2m-3} \rangle \subset K[z_1, \ldots, z_{2m-2}, u_1, \ldots, u_a] $ and 
	where we have $ z_1 z_2 = 0 $ or $ \rho $ is of the form as in $(ii)$ of Lemma~\ref{L:auxSing}.
	\\
	We blow up the origin, which is the intersection of all irreducible components in $ D $. 
	Since $ C \nsubseteq D $, the strict transform $ C ' $ of $ C $ must appear in one of the charts.  
	Since $ ( h_1, \ldots, h_{n-2}) $ is a {Gr\"obner} basis of the ideal $ I $,
	the strict transform of $ I $ is generated by their strict transforms $ h_1', \ldots, h_{n-2}'  $. 
	We go through the different charts of the blowup.
	
	\noindent
	\underline{$ Z_{2k-1} $-chart, $ k \in \{ 2, \ldots, n- 2\} $.}
	Without loss of generality, we assume $ k = 2 $. 
	We have $ z_i = z_3 z_i $ for every $ i \neq 3 $. 
	(By abuse of notation, we denote the variables after the blowup by the same letter.)
	Hence, we get
	\[ 
	\begin{array}{l}
	\displaystyle  
		h_1' 
		= z_1 z_2 (1  - z_3^2 z_{2n-3} z_{2n-2}   + z_3^{n-2} \prod_{\ell=1}^{n-1} z_{2\ell-1} )
	+ z_{2n-3} z_{2n-2}  \ ,
	\\
		h_2' = z_1 z_2 - z_4 \ , 
	\\[5pt]
		h_k' = z_1 z_2 - z_{2k-1} z_{2k} 
		\  , \ \ \ \mbox{ for } k \in \{ 3, \ldots, n-2 \} \ .
	\end{array} 
	\]	
	Since $ z_4 $ appears only in $ h_2' $, we can eliminate it and forget the generator $ h_2' $ without changing the other $ h_k' $, $ k \geq 3 $.
	Notice that $ h_1' $ is of the form as $ h $ in Lemma~\ref{L:auxSing}
	and
	thus, we can apply the induction hypothesis for 
	\[
		K[z_1, z_2, z_5, z_6, \ldots, z_{2n-2}, u_1 ]/ \langle h_1', h_3', \ldots, h_{n-2}' \rangle 
		\ ,
	\]
	where $ u_1 := z_3 $. 
	Therefore, the corresponding singular locus is equal to the strict transform of $ D $.
	In particular, $ C' $ has to be empty in this chart.

	\smallskip

	\noindent
	\underline{$ Z_{2k} $-chart, $ k \in \{ 2, \ldots, n-2\} $.}	
	Without loss of generality, we choose $ k = 2 $. 
	We get $ z_i = z_4 z_i $ for every $ i \neq 3 $. 	
	The strict transforms of $ h_1, \ldots, h_{n-2} $ are
	\[ 
	\begin{array}{l}
	\displaystyle  
		h_1' 
		= z_1 z_2 (1  - z_4^2 z_{2n-3} z_{2n-2}   + z_4^{n-1} \prod_{\ell=1}^{n-1} z_{2\ell-1} )
		+ z_{2n-3} z_{2n-2}  \ ,
		\\
		h_2' = z_1 z_2 - z_3 \ , 
		\\[5pt]
		h_k' = z_1 z_2 - z_{2k-1} z_{2k} 
		\  , \ \ \ \mbox{ for } k \in \{ 3, \ldots, n-2 \} \ .
	\end{array} 
	\]		
	We eliminate $ h_2' $ by replacing $ z_3 = z_1 z_2 $. 
	Observe that this changes $ h_1' $ as $ z_3 $ appears in the product. 
	Since we have already treated the $ Z_3 $-chart,
	it is sufficient to consider only those points of the $ Z_4 $-chart, which are not contained in the $ Z_3 $-chart. 
	Therefore, the singular points, which we have to determine here, fulfill the extra condition $ z_3 = 0 $.
	(Using the precise distinction of the variables before and after the blowup \change{as discussed in the proof of Proposition~\ref{Prop:Cn_p2}},
	we have $ z_3 = z_4 z_3' $, where $ z_3' := \frac{Z_3}{Z_4} $,
	and hence, we avoid the chart $ Z_3 \neq 0 $ by setting $ Z_3 = 0 $, which leads to $ z_3' = 0 $).
	The relation $ z_3 = z_1 z_2 $ implies that we must have $ z_1 z_2 = 0 $, which is {Lemma~\ref{L:auxSing}(1)}. 
	Therefore, we can apply the induction hypothesis and obtain $ C' = \varnothing $ in the $ Z_4 $-chart.

	\smallskip

	\noindent
	\underline{$ Z_1 $-chart.} 	
	Here, $ z_i = z_1 z_i $ for every $ i \neq 1 $ and we obtain
	\[ 
	\begin{array}{l}
		\displaystyle  
		h_1' 
		= z_2 (1  - z_1^2 z_{2n-3} z_{2n-2}   + z_1^{n-2} \prod_{\ell=1}^{n-1} z_{2\ell-1} )
		+ z_{2n-3} z_{2n-2}  \ ,
		\\
		h_2' = z_2 - z_3 z_4 \ , 
		\\[5pt]
		h_k' = z_2 - z_{2k-1} z_{2k} 
		\  , \ \ \ \mbox{ for } k \in \{ 3, \ldots, n-2 \} \ .
		\end{array} 
	\]			
	We replace $ z_2 = z_3 z_4 $ and drop $ h_2 ' $ in the list of generators.
	This provides
	\[ 
	\begin{array}{l}
	\displaystyle  
	h_1' 
	= z_3 z_4 (1  - z_1^2 z_{2n-3} z_{2n-2}   + z_1^{n-2} \prod_{\ell=1}^{n-1} z_{2\ell-1} )
	+ z_{2n-3} z_{2n-2}  \ ,
	\\
	h_k' = z_3 z_4 - z_{2k-1} z_{2k} 
	\  , \ \ \ \mbox{ for } k \in \{ 3, \ldots, n-2 \} \ .
	\end{array} 
	\]			
	Again, we can apply the induction hypothesis using $ u_1 := z_1 $
	and the strict transform of $ C $ is empty in this chart.  
	
	Combining the arguments of the $ Z_{1} $- and the $ Z_4 $-charts shows that $ C' = \varnothing $ in the $ Z_2 $-chart.

	\smallskip

	\noindent
	\underline{$ Z_{2n-3} $-chart.} 	
	We have $ z_i = z_{2n-3} z_i $ for every $ i \neq 2n-3 $.
	The strict transforms of $ h_1, \ldots, h_{n-2} $ are 
	\[ 
	\begin{array}{l}
	\displaystyle  
	h_1' 
	=  z_1 z_2 (1  - z_{2n-3}^2 z_{2n-2}  + z_{2n-3}^{n-1} \prod_{\ell=1}^{n-1} z_{2\ell-1} )
	+ z_{2n-2}  \ ,
	\\
	h_k' = z_1 z_2 - z_{2k-1} z_{2k} 
	\  , \ \ \ \mbox{ for } k \in \{ 2, \ldots, n-2 \} \ .
	\end{array} 
	\]
	Since we already handled the $ Z_i $-charts for $ i \in \{ 1, \ldots, 2n-4 \} $,
	we only have to take those points into account which are not contained in these charts. 
	Therefore, 
	analogous to the $ Z_4 $-chart, it suffices if we determine only those singular points, for which we have
	additionally $ z_1 = \ldots = z_{2n-4} = 0 $.
	
	We may rewrite $ h_1' $ as 
	\[ 
		z_{2n-2} ( 1 - z_1 z_2 z_{2n-3}^2 ) + z_1 z_2 (1  + z_{2n-3}^{n-1} \prod_{\ell=1}^{n-1} z_{2\ell-1} )
	\]
	and, by the previous, $ 1 - z_1 z_2 z_{2n-3}^2 $ is a unit.
	This implies that we may eliminate $ z_{2n-2} $ and forget $ h_1' $.
	The resulting ideal is binomial.
	In particular, we can apply the induction hypothesis with $ u_1 := z_{2n-3} $ 
	(as $ \rho = 0 $ is a possible choice)
	and we get $ C' = \varnothing $ in the present chart.

	The analogous arguments can be applied for the remaining $ Z_{2n-2} $-chart.
	This concludes the proof that \eqref{eq:D_Sn} is an equality.
	
	The results on the desingularization, 
	on the type of the singularity at a generic point of an irreducible component of $ \Sing ( \Spec(\mc{A}(\St_n))) $,
	and on the local description of $ \change{\Spec(\mc{A}(\St_n))} $ at the origin
	follow:
	In each of the charts above, we blow up the intersection of the irreducible components of strict transform $ D $ and continue this process until we eventually reach the case,
	where there is only one irreducible component left. 
	Since we eliminate after every blowup one generator,
	the strict transform of the {variety} is isomorphic to a hypersurface as in Lemma~\ref{L:auxSing} in every chart.
	In particular, we get a hypersurface singularity of type $ A_1 $ and 
	all singularities are resolved after the next blowup. 	
	\\
	Finally, after localizing at $ \langle z_1, \ldots, z_{2n-2} \rangle $,  the factor in parentheses of $ h_1 $ becomes a unit,
		which we abbreviate as $ \epsilon $.
		Therefore, we may introduce $ x_{2n-2} :=  \epsilon^{-1} z_{2n-2} $ 
		and the ideal generated by $ h_k $, 
		for $ k \in \{ 1, \ldots, n-2 \} $, 
		is binomial.
\end{proof}

\begin{Bem} \label{R:Starcomp}
	As we have seen and using Remark~\ref{Rk:Sn_known} for $ n \in \{ 2,3 \} $, the number of irreducible components in the singular locus of $ \Spec(\cA(\St_n)) $ is $ s(n) := (n-1)(n-2)2^{n-4} $, for $ n \in \{ 2, 3, 4,5, \ldots \}  $;
	more concretely, $ s(2) = 0, s(3) = 1, s(4)=6, s(5)=24, s(6)=80 $, \ldots .
	This integer sequence appears in
	{\em The On-Line Encyclopedia of Integer Sequences}, 
	\cite[Sequence A001788]{OEIS}.
	There, the sequence is $ a(n) : =n(n+1)2^{n-2} = s(n+2)$,
	for $ n \in \ZZ_{\geq 0} $.
	\\
	One of the provided descriptions is the following: 
	Let $ X $ be a set with $ 2n $ elements and let $ X_1, \ldots, X_n $ be a partition of $ X $ into $ 2 $-blocks.
	For $ n > 1 $, the number $ a(n-1) $ coincides with the number subsets of $ X $ with $ n +2 $ elements
	and which intersect every $ X_i $ for $ i \in \{ 1, \ldots, n \} $. 
	This precisely describes how we obtained the components of the singular locus.
	Let us briefly explain this (for the case $ \St_{n+1} $):
	\begin{itemize}
		\item 
		The set $ X := \{ z_1, \ldots, z_{2n} \} $ is the set of variables.
		
		\item 
		The partition in $ 2 $-blocks is determined by the monomials appearing in $ h_1, \ldots, h_{n-1} $,
		namely, the blocks are $ X_i := \{ z_{2i-1}, z_{2i} \} $, for $ i \in \{ 1, \ldots, n \} $.
		
		\item 
		A subset with $ n+2 $ elements determines a {regular} $ (n-2) $-dimensional subvariety $ C \subset  \A_K^{2n} $.
		The condition that every $ X_i $ has to be intersected ensures that we have 
		$ C \subset \Spec(\cA(\St_{n+1})) $.
		Furthermore, the number of elements $ n + 2 $ provides that there is a generator which is singular at $ C $ and 
		that 
		$ C \subset \Sing(\Spec(\cA(\St_{n+1}))  $ is an irreducible component.
 	\end{itemize}
\end{Bem}

\bigskip
\newcommand{\etalchar}[1]{$^{#1}$}
\def\cprime{$'$}

\end{document}